\documentclass[12pt]{amsart}
\usepackage{amscd}     
\usepackage{amssymb}
\usepackage{amsmath, amsthm, graphics}
\usepackage[all]{xy}
\xyoption{2cell} \UseAllTwocells \xyoption{frame} \CompileMatrices
\allowdisplaybreaks[3]
\usepackage{amsfonts}
\usepackage[normalem]{ulem}
\usepackage{hyperref}
\usepackage{tikz}
\usepackage{mathtools}
\usepackage{verbatim}

\usepackage{latexsym}
\usepackage{epsfig}
\usepackage{amsfonts}
\usepackage{enumerate}
\usepackage{times,float}
\usepackage{graphics}

\newtheorem{prop}{Proposition}
\newtheorem{lem}[prop]{Lemma}

\newtheorem{cor}[prop]{Corollary}
\newtheorem{thm}[prop]{Theorem}
\newtheorem{rem}[prop]{Remark}

\newtheorem{question}[prop]{Question}
\newtheorem{conjecture}[prop]{Conjecture}

\newtheorem{defn}[prop]{Definition}

\newtheorem{notation}[prop]{Notation}

\theoremstyle{remark}

\theoremstyle{remark}

\numberwithin{equation}{section}

\newcommand{\X}{\mathcal{X}}
\newcommand{\Y}{\mathcal{Y}}
\newcommand{\Z}{\mathcal{Z}}

\newcommand{\M}{\mathcal{M}}
\newcommand{\C}{\mathcal{C}}

\newcommand{\D}{\mathcal{D}}

\newcommand{\bt}{\mathbf{t}}

\newcommand{\on}{\operatorname}
\newcommand{\bM}{\overline{\mathcal M}}

\def\<{\left\langle}
\def\>{\right\rangle}

\def\b1{{\mathbf 1}}

\begin{document}

\title{A Gromov-Witten theory for simple normal-crossing pairs without log geometry}
\author{Hsian-Hua Tseng}
\address{Department of Mathematics\\ Ohio State University\\ 100 Math Tower, 231 West 18th Ave.\\Columbus\\ OH 43210\\ USA}
\email{hhtseng@math.ohio-state.edu}

\author{Fenglong You}
\address{Department of Mathematics, University of Oslo, PO Box 1053 Blindern,
0316 Oslo, Norway}
\email{youf@math.uio.no}

\date{\today}

\begin{abstract}
    We define a new Gromov-Witten theory relative to simple normal crossing divisors as a limit of Gromov-Witten theory of multi-root stacks. Several structural properties are proved including relative quantum cohomology, Givental formalism, Virasoro constraints (genus zero) and a partial cohomological field theory. Furthermore, we use the degree zero part of the relative quantum cohomology to provide an alternative mirror construction of Gross-Siebert \cite{GS19} and to prove the Frobenius structure conjecture of Gross-Hacking-Keel \cite{GHK} in our setting.
\end{abstract}

\maketitle

\tableofcontents

\section{Introduction}
\subsection{The theory}
Let $X$ be a smooth projective variety over $\mathbb{C}$ and let $$D_1, ..., D_n\subset X$$ be smooth irreducible divisors. Suppose $$D:=D_1+...+D_n$$ is simple normal crossing.

For $r_1, ..., r_n\in \mathbb{N}$ pairwise coprime, the multi-root stack $$X_{D, \vec{r}}:=X_{(D_1, r_1),...,(D_n, r_n)},$$ defined in Definition \ref{defn-multi-root}, is smooth. The first result of this paper shows that the Gromov-Witten theory of $X_{D, \vec{r}}$ is a polynomial in $r_1,...,r_n$, see Corollary \ref{cor-limit} in Section \ref{sec:application}. This is achieved by certain polynomiality results for root stacks associated to a pair $(\X,\D)$ of Deligne-Mumford stack $\X$ and a smooth divisor $\D\subset \X$.
\begin{thm}
For $r$ sufficiently large, genus $0$ Gromov-Witten invariant of $\X_{\D,r}$ is independent of $r$. Genus $g>0$ Gromov-Witten invariant of $\X_{\D, r}$ is a polynomial in $r$. Furthermore, the constant term of the polynomial is the corresponding relative Gromov-Witten invariant of $(\X,\D)$.
\end{thm}
We refer the readers to Theorems \ref{thm:poly} and \ref{thm:poly-neg} in Section \ref{sec:polynomiality} for the precise statement.  Taking the constant terms yields a theory canonically attached to the pair $(X, D)$. See Definition \ref{defn:GW_inv} in Section \ref{sec:application} for the precise definition of this new theory. 

We may view this new theory formally as the Gromov-Witten theory of the {\em infinite root stack} $$X_{D, \infty}$$  associated to $(X, D)$, as constructed in \cite{TV}, because in genus $0$ we show that the Gromov-Witten theory of $X_{D, \vec{r}}$ is independent of $r_1,...,r_n$ and taking constant terms is the same as taking large $r_i$ limit.
\begin{question}
Can one define Gromov-Witten theory of infinite root stacks directly? 
\end{question}

Naturally, one can expect such a  definition to coincide with the constant terms of Gromov-Witten theory of finite root stacks. By \cite{TV}, the infinite root stack structure determines the logarithmic structure. It is natural to expect that infinite root stack Gromov-Witten theory should determine logarithmic Gromov-Witten theory.

\subsection{Logarithmic theory}
Our new theory has some advantages:
\begin{enumerate}
\item {\em Negative contact orders are naturally included.} 
A relative marking with positive contact order $k>0$ along a divisor $D_i$ corresponds to an orbifold marking with $age (N_{D_i/X_{D,\vec r}})$ equals to $k/r_i$  for $r_i\gg 1$. On the other hand, a relative marking with negative contact order $k<0$ along a divisor $D_i$ comes from an orbifold marking with $age (N_{D_i/X_{D,\vec r}})$ equals to $1+k/r_i$  for $r_i\gg 1$. Roughly speaking, if we have negative contact order with a divisor $D_i$ at a marking, then the irreducible component of the curve containing this marking should map into $D_i$. When $D$ is irreducible, we recover relative Gromov-Witten theory with negative contact orders defined in \cite{FWY} and \cite{FWY19} which is a generalization of the usual relative Gromov-Witten theory of \cite{LR}, \cite{IP}, \cite{Li01} and \cite{Li02}

    \item {\em It enjoys very nice properties.} In particular, we highlight the following properties.
\begin{itemize}
    \item In genus zero, we have
\begin{itemize}
    \item Topological recursion relation (TRR) (Section \ref{Sec:univ_eqns})
    \item Witten-Dijkgraaf-Verlinde-Verlinde (WDVV) equation (Section \ref{Sec:univ_eqns})
    \item Relative quantum cohomology ring (Section \ref{sec:rel-qh})
    \item Givental formalism (Section \ref{sec:givental})
    \item Virasoro constraints (Section \ref{Sec:vir}).
\end{itemize}
\item In all genera, we have
\begin{itemize}
    \item string, dilaton, and divisor equations (Section \ref{Sec:univ_eqns})
    \item a Partial CohFT (Section \ref{sec:CohFT}).
\end{itemize}
\end{itemize}
    
    \item {\em It is quite computable.} It has already been proved in \cite{TY20b} that one can construct an $I$-function for the Gromov-Witten theory of $X_{D,\infty}$. Therefore, Givental formalism that we developed in Section \ref{sec:givental} provides a necessary foundation for \cite{TY20b} to state a mirror theorem for $X_{D,\infty}$ (see Theorem \ref{thm:mirror}). The mirror theorem allows us to compute genus zero invariants of $X_{D,\infty}$ in various cases. Some examples and applications were given in \cite{TY20b}. Therefore, one may expect that Gromov-Witten invariants of infinite root stacks are more accessible (than log Gromov-Witten invariants) in terms of computation, as lots of sophisticated techniques in traditional Gromov-Witten theory are available.

\end{enumerate}

We may view our new theory as a logarithmic Gromov-Witten theory of $(X, D)$. As such, it is natural to ask 
\begin{question}
How is the new theory related to the (punctured) logarithmic Gromov-Witten theory of Abramovich-Chen-Gross-Siebert defined in \cite{GS13}, \cite{C}, \cite{AC14}, \cite{ACGS}?
\end{question}

In \cite{TY20b}, we showed by explicit computations that these two theories are equal in some cases. When $D$ is irreducible, the main results of \cite{ACW} and \cite{TY} imply that these two theories are the same for invariants without punctured points\footnote{The arguments easily extend to the case $D_i$'s are disjoint, showing that the two theories are the same in this case, too.}. As pointed out by Dhruv Ranganathan, these two theories are not equal in general. For example, logarithmic invariants are invariant uner birational transformation \cite{AW}, but orbifold invariants are not.
However, it is perhaps reasonable to expect that our new theory and the punctured logarithmic Gromov-Witten theory are equivalent somehow. It would be interesting to find the precise relation between these two theories. Then, one can compute punctured invariants through corresponding invariants of $X_{D,\infty}$. Recently, the birational invariance of orbifold invariants has been investigated in \cite{BNR} and \cite{You22}. 

Another interesting question is 
\begin{question}[R. Pandharipande]
Does the new theory have a {\em degeneration formula}?
\end{question}
When $D$ is irreducible and there are no punctured points, it is proved in \cite{TY} that our theory is the relative Gromov-Witten theory of \cite{Li01}, which admits a  degeneration formula \cite{Li02}. A degeneration formula for logarithmic Gromov-Witten theory can be found in \cite{ACGS17}, \cite{KLR} and \cite{R}.

\subsection{Mirror constructions}

In \cite{GS18} and \cite{GS19}, Gross-Siebert constructed mirrors to a log Calabi-Yau pair $(X,D)$ and a maximally unipotent degeneration $X\rightarrow S$ of log Calabi-Yau manifolds. The mirrors are constructed from the degree $0$ part of the relative quantum cohomology ring $$QH^0(X,D).$$ A key ingredient is the punctured Gromov-Witten theory which is used to describe the structure constants for the product rule. 

We construct a relative quantum cohomology ring for the pair $(X,D)$ in Section \ref{sec:rel_QC} using Gromov-Witten invariants of $X_{D,\infty}$. The associativity of the relative quantum cohomology follows from the WDVV equation. Restricting it to the degree $0$ part of the relative quantum cohomology ring, $$QH^0(X_{D,\infty}),$$ there is a product structure naturally coming from the restriction of the relative quantum product. Similar to \cite{GS19}, the associativity is not expected to be preserved under this restriction. We show in Section \ref{sec:frob} that the associativity is true under some assumptions. More precisely, we have the following theorem.

\begin{thm}[=Theorem \ref{thm-stru-const}]\label{thm-intro-stru-const}
When $(K_X+D)$ is nef or anti-nef, the structure constants $$N^{\on{orb},\beta}_{p_1,p_2,-r}$$ define, via (\ref{equ-prod}), a commutative, associative $S_I$-algebra structure on $R_I$ with unit given by $\vartheta_0$,  where $S_I$ and $R_I$ are defined in (\ref{S-I}) and (\ref{R-I}) respectively; the structure constants are defined in (\ref{def-stru-const}).
\end{thm}

\begin{rem}
Theorem \ref{thm-intro-stru-const} is \cite[Theorem 1.9]{GS19}, which is a main theorem of \cite{GS19}, if we replace the structure constants by the corresponding punctured Gromov-Witten invariants. It is worth noting that in our setting the proof of the associativity is substantially shorter. Gross--Siebert also proved the case when $(X,D)$ is (non-minimal) log Calabi-Yau in \cite[Theorem 1.12]{GS19}, which would avoid issues from the existence of minimal models. We plan to study this case in the future.
\end{rem}

Furthermore, we show that the Frobenius structure conjecture of Gross-Hacking-Keel \cite{GHK} holds in our setting.
\begin{thm}[=Theorem \ref{thm-frob}]
When $(K_X+D)$ is nef or anti-nef, the Frobenius structure conjecture
(see Conjecture \ref{conj-frob}) holds for $QH^0(X_{D,\infty})$.
\end{thm}

In Section \ref{sec:mirror-constr}, we use the algebra in Theorem \ref{thm-intro-stru-const} to construct mirrors following the Gross-Siebert program (see \cite{GS18} and \cite{GS19}). Naturally, one can ask
\begin{question}
How are the resulting mirrors related to mirrors from other constructions?
\end{question}

One can expect that the resulting mirrors are closely related to, if not the same as, Gross-Siebert mirrors. One such evidence is given in \cite[Section 6]{TY20b} where we obtained a mirror identity between quantum periods of Fano varieties and classical periods of their mirror Landau-Ginzburg potentials by replacing log invariants with formal invariants of infinite root stacks.

\subsection{Acknowledgement}
 We thank Mark Gross, Rahul Pandharipande, Dhruv Ranganathan, and Helge Ruddat for valuable comments and suggestions.
 
 H.-H. T. is supported in part by Simons foundation collaboration grant.  F. Y. is supported by a postdoctoral fellowship funded by NSERC and Department of Mathematical Sciences at the University of Alberta.

\section{Polynomiality}\label{sec:polynomiality}
In this section, we generalize the main results of \cite{TY}, \cite{FWY} and \cite{FWY19} to the case when the target $\X$ is a Deligne-Mumford stack instead of a variety. In the next section, we will use these results to prove the polynomiality of Gromov-Witten theory of multi-root stacks.

\subsection{Set-up}
Let $\X$ be a smooth proper Deligne-Mumford stack over $\mathbb{C}$ with projective coarse moduli space. Let $$\D\subset \X$$ be a smooth irreducible divisor. Assume that $r\in \mathbb{N}$ is coprime with the order of any stabilizer of $\X$. Then the stack of $r$-th roots along $\D$, $$\X_{\D,r},$$ is smooth and we consider its Gromov-Witten theory.

Given an effective curve class $\beta \in H_2(\X,\mathbb Q)$, let 
\[
\vec k=(k_1,\ldots,k_m)\in (\mathbb Q^\times)^m
\]
be a vector that satisfies  
\[
\sum_{j=1}^m k_j=\int_\beta [\D]. 
\]

The number of positive and negative elements in $\vec k$ are denoted by $m_+$ and $m_-$ respectively. So
\[
m=m_++m_-.
\]

We assume that $r$ is sufficiently large. We consider the moduli space $$\bM_{g,\vec k,n}(\X_{\D,r},\beta)$$ of $(m+n)$-pointed, genus $g$, degree $\beta \in H_2(\X,\mathbb Q)$ orbifold stable maps to $\X_{\D,r}$ where the $j$-th marking is an orbifold marking with $age (N_{\D/\X})$ equals to $k_j/r$ if $k_j>0$; the $j$-th marking is an orbifold marking with $age (N_{\D/\X})$ equals to $1+k_j/r$ if $k_j<0$; there are $n$ extra markings that map to $\underline{\mathcal I}\X$, the rigidified inertia stack of $\X$. We consider the forgetful map
\[
\tau_{\on{orb}}:\bM_{g,\vec k,n}(\X_{\D,r},\beta) \rightarrow \bM_{g,m+n}(\X,\beta)\times_{(\underline{\mathcal I}\X)^m} (\underline{\mathcal I}\D)^m.
\]

We first consider the case when $m_-=0$, namely, there are only positive contact orders. In this case, we write $$\bM_{g,\vec k,n}(\X/\D,\beta)$$ for the corresponding moduli space of relative orbifold stable maps to $(\X,\D)$ where the contact orders are given by $\vec k$. We consider the forgetful map
\[
\tau_{\on{rel}}:\bM_{g,\vec k,n}(\X/\D,\beta) \rightarrow \bM_{g,m+n}(\X,\beta)\times_{(\underline{\mathcal I}\X)^m} (\underline{\mathcal I}\D)^m.
\]

\begin{thm}\label{thm:poly}
For $m_-=0$ and $r$ sufficiently large, genus $0$ Gromov-Witten invariant of $\X_{\D,r}$ is independent of $r$. Genus $g>0$ Gromov-Witten invariant of $\X_{\D, r}$ is a polynomial in $r$. Furthermore, the constant term of the polynomial is the corresponding relative Gromov-Witten invariant of $(\X,\D)$. More precisely, we have the following results at the cycle level.
\[
\left[(\tau_{\on{orb}})_*\left[\bM_{g,\vec k,n}(\X_{\D,r},\beta)\right]^{\on{vir}}\right]_{r^0}=(\tau_{\on{rel}})_*\left[\bM_{g,\vec k,n}(\X/\D,\beta)\right]^{\on{vir}}
\]
and 
\[
(\tau_{\on{orb}})_*\left[\bM_{0,\vec k,n}(\X_{\D,r},\beta)\right]^{\on{vir}}
\]
is independent of $r$, where $[\cdots]_{r^0}$ means the constant term of a polynomial in $r$.
\end{thm}

\begin{thm}\label{thm:poly-neg}
For $m_->0$ and $r$ sufficiently large, after multiplying by $r^{m_-}$, genus $0$ Gromov-Witten invariant of $\X_{\D,r}$ is independent of $r$. After multiplying by $r^{m_-}$, genus $g>0$ Gromov-Witten invariant of $\X_{\D, r}$ is a polynomial in $r$. More precisely, 
\[
r^{m_-}(\tau_{\on{orb}})_*\left[\bM_{g,\vec k,n}(\X_{\D,r},\beta)\right]^{\on{vir}}
\]
is a polynomial in $r$ and 
\[
r^{m_-}(\tau_{\on{orb}})_*\left[\bM_{0,\vec k,n}(\X_{\D,r},\beta)\right]^{\on{vir}}
\]
is independent of $r$.
\end{thm}

\begin{rem}
The degree of this polynomial can be studied using the method of \cite{TY20}. One can show that the degree of this polynomial is bounded by $2g-1$ for $g\geq 1$. Since we do not use such a result, we leave the proof to the interested readers.
\end{rem}

\begin{rem}
Theorem \ref{thm:poly-neg} generalizes the main result of \cite{FWY} and \cite{FWY19} to the orbifold case, namely $\X$ is a Deligne-Mumford stack instead of a variety. Therefore, we can also define relative Gromov-Witten theory of $(\X,\D)$ with negative contact orders as a limit of orbifold Gromov-Witten theory of $\X_{\D,r}$. Similar to \cite{FWY} and \cite{FWY19}, with some extra work, we can define relative Gromov-Witten theory of $(\X,\D)$ with negative contact orders purely in terms of relative Gromov-Witten theory of $(\X,\D)$ with positive contact orders and rubber theory of $\D$. 
\end{rem}

\begin{rem}
There are some immediate applications of Theorem \ref{thm:poly} and Theorem \ref{thm:poly-neg}. First of all, the genus zero case has been used in \cite{You19} to compute genus zero relative invariants of certain compactifications of toric Calabi-Yau orbifolds which coincide with some genus zero open invariants of toric Calabi-Yau orbifolds. These invariants are precisely instanton corrections of the mirror of toric Calabi-Yau orbifolds. Moreover, a sketch of the proof of Theorem \ref{thm:poly} is given in \cite[Appendix A]{You19}. Secondly, it has been used to deduce the gerbe duality for relative Gromov-Witten theory from absolute Gromov-Witten theory, see \cite{TT19}.
\end{rem}

\subsection{Proof of Theorem \ref{thm:poly}}
Following the strategy of \cite{TY}, to analyze the $r$-dependence of Gromov-Witten invariants of $\X_{\D,r}$, we use the degeneration formula to reduce to a local model. We also refer to \cite[Section 4.2]{FWY19} for some details.

\subsubsection{Degeneration}
Let $$p:\mathfrak{X}\to \mathbb{A}^1$$ be the deformation to the normal cone of $\D\subset \X$. The special fiber $p^{-1}(0)$ is $\X$ and $$\Y:=\mathbb{P}(N_{\D/\X}\oplus\mathcal{O}_\X)$$ glued together by identifying $\D\subset \X$ with 
$$\D_\infty:=\mathbb{P}(N_{\D/\X}) \subset \mathbb{P}(N_{\D/\X}\oplus\mathcal{O}_\X).$$
Other fibers $p^{-1}(t\neq 0)$ are isomorphic to $\X$. There is a divisor $$\mathfrak{D}\subset \mathfrak{X}$$ whose restriction to $p^{-1}(t\neq 0)$ is $\D$ and whose restriction to $p^{-1}(0)$ is 
$$\D_0:= \mathbb{P}(\mathcal{O}_\X)\subset  \mathbb{P}(N_{\D/\X}\oplus\mathcal{O}_\X).$$
The $r$-th root stack of $\mathfrak{X}$ along $\mathfrak{D}$, $$\mathfrak{X}_{\mathfrak{D}, r},$$ is a flat degeneration of $\X_{\D, r}$ to $$\X\cup_{\D=\D_\infty}\mathbb{P}(N_{\D/\X}\oplus \mathcal{O}_\X)_{\D_0, r}.$$

The degeneration formula for orbifold Gromov-Witten theory \cite{AF} applied to $\mathfrak{X}_{\mathfrak{D}, r}$ expresses Gromov-Witten invariants of $\X_{\D, r}$ in terms of (disconnected) relative Gromov-Witten invariants of $(\X, \D)$ and $(\mathbb{P}(N_{\D/\X}\oplus \mathcal{O}_\X)_{\D_0, r}, \D_\infty)$. The sum in the degeneration formula ranges over the intersection profile along $\D$. Since $(\X, \D)$ is independent of $r$, the $r$-dependence must come from orbifold-relative Gromov-Witten invariants of $(\Y_{\D_0, r}=\mathbb{P}(N_{\D/\X}\oplus \mathcal{O}_\X)_{\D_0, r}, \D_\infty)$. Therefore, we just need to compute
\begin{align}\label{push-forward-cycle}
(\tau^\prime)_*\left[\bM_{g, \vec k,n,\vec \mu}(\Y_{\D_0,r}/\D_\infty,\beta)\right]^{\on{vir}},
\end{align}
where  $\vec \mu\in (\mathbb Z_{>0})^{|\vec \mu|}$ records contact orders at $\D_\infty$ and $\tau^\prime$ is the forgetful map
\[
\tau^\prime: \bM_{g, \vec k,n,\vec \mu}(\Y_{\D_0,r}/\D_\infty,\beta) \rightarrow \bM_{g, m+n+|\vec \mu|}(\D,\beta).
\]

\subsubsection{Localization}\label{sec:localization}
The orbifold-relative Gromov-Witten theory of $(\Y_{\D_0, r}, \D_\infty)$ may be studied using virtual localization with respect to the $\mathbb{C}^*$-action that scales the fibers of $\Y_{\D_0,r}\to \D$.

When $\D$ is a scheme and $r$ is sufficiently large, the virtual localization formula has been written in detail in \cite{JPPZ2} and \cite{TY}. In the present case the formula is completely analogous. We write $\sqrt[r]{L/\D}$ for the $r$-th root of the line bundle $L$ over $\D$. Recall that $\sqrt[r]{L/\D}$ is a gerbe over $\D$ banded by $\mu_r$. The virtual localization formula expresses (\ref{push-forward-cycle}) as a sum over decorated graphs. For the purpose of analyzing the $r$-dependence, we only need to note that $r$ only appears  in the contribution from stable vertices $v$ over $\D_0$, given by the following expression capping with the virtual class $[\overline{\mathcal{M}}_{g(v),n(v)}(\sqrt[r]{L/\D}, \beta(v))]^{\on{vir}}$:
\begin{align}\label{eqn:r_dep}
&\left(\prod_{e\in E(v)}\frac{|G_{(e,v)}|}{r_{(e,v)}}\frac{r_{(e,v)}d_e}{t+\on{ev}_{e}^*c_1(L)-d_e\bar{\psi}_{(e,v)}}\right)\cdot\left(\sum_{i=0}^{\infty}(t/r)^{g(v)-1+|E(v)|-i}c_i(-R^\bullet\pi_*\mathcal L)\right)\\
\notag =& t^{-1}\left(\prod_{e\in E(v)}\frac{|G_{(e,v)}'|}{1}\frac{d_e}{1+(\on{ev}_{e}^*c_1(L)-d_e\bar{\psi}_{(e,v)})/t}\right)\cdot\left(\sum_{i=0}^{\infty}t^{g(v)-i}(r)^{i-g(v)+1}c_i(-R^\bullet\pi_*\mathcal L)\right)\\
\notag=& t^{-1}\left(\prod_{e\in E(v)}\frac{|G_{(e,v)}'|}{1}\frac{d_e}{1+(\on{ev}_{e}^*c_1(L)-d_e\bar{\psi}_{(e,v)})/t}\right)\cdot\left(\sum_{i=0}^{\infty}(tr)^{g(v)-i}(r)^{2i-2g(v)+1}c_i(-R^\bullet\pi_*\mathcal L)\right),
\end{align}
where 
\begin{itemize}
    \item $g(v)$ is the genus of the vertex $v$ over $\D_0$ in a localization graph,
    
    \item $n(v)$ is the number of marked points of the vertex $v$,
    
    \item $\beta(v)$ is the degree assigned to the vertex $v$,
    
    \item $t$ is the equivariant parameter,
    
    \item $L=N_{\D/\X}$,
    \item \[
\pi: \mathcal C_{g(v),n(v)}(\sqrt[r]{L/\D}, \beta(v))\rightarrow \overline{\mathcal{M}}_{g(v),n(v)}(\sqrt[r]{L/\D}, \beta(v))
\]
 is the universal curve, 
\[
\mathcal L\rightarrow \mathcal C_{g(v),n(v)}(\sqrt[r]{L/\D}, \beta(v))
\] 
is the universal $r$-th root,
    
    \item $d_e$ is the degree of the edge $e\in E(v)$,
    \item $\text{ev}_e$ is the evaluation map at the node corresponding to $e$,
    \item $\bar{\psi}_{(e,v)}$ is the descendant class at the marked point corresponding to the pair $(e,v)$,
    
    \item  $G_{(e,v)}$ 
    is the stabilizer group associated to the vertex $v$ and the edge $e$. The group $G_{(e,v)}$ is a $\mu_r$ extension of $G_{(e,v)}'$, so $$|G_{(e,v)}|=r|G_{(e,v)}'|.$$ The group $G_{(e,v)}'$
    is independent of $r$.
    \item $r_{(e,v)}$ is the order of the orbifold structure at the node indexed by $(e,v)$.

\end{itemize}

Moreover, if the target expands over $\D_\infty$, the vertex contribution over $\D_\infty$ is
\begin{align}\label{contri-infinity}
\left(\prod_{e\in E(v)}\frac{|G_{(e,v)}|}{r_{(e,v)}}\right)\frac{\prod_{e\in E(\Gamma)}d_e r_{(e,v)}}{t+\psi_\infty},
\end{align}
which 
always contribute to negative powers of $t$. The edge contribution is trivial when $r$ is sufficiently large.

To obtain genus $g$ Gromov-Witten invariants of $(\Y_{\D_0, r}, \D_\infty)$, we must take the non-equivariant limit, i.e. taking the $t^0$ coefficient in the localization formula. 

If the genus $g=0$, then $g(v)=0$ and we note that (\ref{eqn:r_dep}) and (\ref{contri-infinity}) only contain negative powers of $t$. It follows by the arguments of \cite[Lemma 4.8]{FWY19} that the $t^0$ coefficient is $0$ unless $\bM_{0, \vec k,n,\vec \mu}(\Y_{\D_0,r}/\D_\infty,\beta)$ is unstable (genus zero, two markings and curve class zero). Then the degeneration formula simplifies to
\[
(\tau_{\on{orb}})_*\left[\bM_{0,\vec k,n}(\X_{\D,r},\beta)\right]^{\on{vir}}=(\tau_{\on{rel}})_*\left[\bM_{0,\vec k,n}(\X/\D,\beta)\right]^{\on{vir}}.
\]

Now we assume $g>0$. 
\begin{prop}\label{prop:poly}
For $r$ sufficiently large and $i\geq 0$, the class $$r^{2i-2g(v)+1}\tau^\prime_*(c_i(-R^*\pi_*\mathcal{L})\cap [\overline{\mathcal{M}}_{g(v),n(v)}(\sqrt[r]{L/\D}, \beta(v))]^{\on{vir}})$$ is a polynomial in $r$. Here $\tau^\prime: \overline{\mathcal{M}}_{g(v),n(v)}(\sqrt[r]{L/\D}, \beta(v))\to \overline{\mathcal{M}}_{g(v),n(v)}(\D, \beta(v))$ is the natural map to the moduli space of stable maps to $\D$.
\end{prop}

The proof of Proposition \ref{prop:poly} will be given in Section \ref{sec:poly}. Here, we complete the proof of the theorem. The polynomiality follows immediately from Proposition \ref{prop:poly}. By the formula (\ref{eqn:r_dep}) and Proposition \ref{prop:poly}, the $t^0r^0$-coefficient of the localization contribution of $(\tau^\prime)_*\left[\bM_{g, \vec k,n,\vec \mu}(\Y_{\D_0,r}/\D_\infty,\beta)\right]^{\on{vir}}$ is $0$ unless $\bM_{g, \vec k,n,\vec \mu}(\Y_{\D_0,r}/\D_\infty,\beta)$ is unstable. Then $r^0$-coefficient of the degeneration formula simplifies to
\[
\left[(\tau_{\on{orb}})_*\left[\bM_{g,\vec k,n}(\X_{\D,r},\beta)\right]^{\on{vir}}\right]_{r^0}=(\tau_{\on{rel}})_*\left[\bM_{g,\vec k,n}(\X/\D,\beta)\right]^{\on{vir}}.
\]

\subsubsection{Proof of Proposition \ref{prop:poly}}\label{sec:poly}
The Chern character $ch(R^\bullet\pi_*\mathcal{L})$ can be calculated explicitly using Toen's Grothendieck-Riemann-Roch formula, see \cite{Tseng}. 
In general, let $\Z$ be a smooth proper Deligne-Mumford stack over $\mathbb{C}$ with projective coarse moduli space, and let $V$ be a line bundle on $\Z$. Consider the universal family $$\pi: \C\to\overline{\mathcal{M}}_{g,n}(\Z, \beta), f: \C\to \Z.$$ A formula for the Chern character $ch(R^\bullet\pi_*f^*V)\cap[\overline{\mathcal{M}}_{g,n}(\Z, \beta)]^{\on{vir}}$ is calculated in \cite{Tseng}. For simplicity, in what follows we omit the capping with virtual classes in the discussion. With this understood, the formula reads
\begin{equation}\label{eqn:GRR}
\begin{split}
    ch(R^\bullet\pi_*f^*V)=&\pi_*(ch(f^*V)Td^\vee(L_{n+1}))\\
    &-\sum_{i=1}^n\sum_{m\geq 1}\frac{ev_i^*A_m}{m!}\psi_i^{m-1}\\
    &+\frac{1}{2}(\pi\circ\iota)_*\sum_{m\geq 2}\frac{1}{m!} r_{node}^2 ev_{node}^*A_m \frac{\psi_+^{m-1}+(-1)^m\psi_-^{m-1}}{\psi_++\psi_-},  
\end{split}    
\end{equation}
where 
\begin{enumerate}
\item 
$Td$ is the Todd class.

\item 
On the component $\Z_i$ of the inertia stack $I\Z$, $A_m$ is $$B_m(age_{\Z_i}(p_i^*V))ch(p_i^*V)=B_m(age_{\Z_i}(p_i^*V))p_i^*(e^{c_1(V)}).$$ Here $p_i: \Z_i\to \Z$ is the natural projection, and $B_m(x)$ are Bernoulli polynomials defined by $$\frac{te^{tx}}{e^t-1}=\sum_{m\geq 0}\frac{B_m(x)}{m!} t^m.$$

\item $\iota$ is the inclusion of the nodal locus into the universal curve $\mathcal{C}$.

\item 
$r_{node}$ is the order of orbifold structure at the node.

\item 
$ev_{node}$ is the evaluation map at the node.

\item 
$\psi_\pm$ are $\psi$ classes associated to branches of the node.

\end{enumerate}

We want to apply the  formula to the case $$\Z=\sqrt[r]{L/\D}$$ the stack of $r$-th roots of the line bundle $L=N_{\D/\X}$ over $\D$, and $V$ the universal $r$-th root line bundle on $\Z$.

For this purpose, we need to discuss how to choose orbifold structures induced from $\Z$ at marked points and nodes.

If a point $p\in \D$ has stabilizer group $G$, then its inverse image $q\in \Z$ has stabilizer group $G(r)$, which is a cyclic extension of $G$: $$1\to \mu_r\to G(r)\to G\to 1.$$
An orbifold structure at a point mapping to $q$ is a conjugacy class of $G(r)$. If the induced orbifold structure at the point (which maps to $p$) is chosen, then this conjugacy class in $G(r)$ can be identified with an element in $\mu_r$. We refer to \cite[Section 3.2]{TY16} for more details.

For the $j$-th marked point from $\overline{\mathcal{M}}_{g, \vec{k}, n}(\Y, \beta)$, the orbifold structure is chosen so that the age of $V$ at this marked point is $k_j/r$ if $k_j\geq 0$ and $1+k_j/r$ if $k_j<0$. For other marked points, which are formed by splitting nodes in $\mathbb{C}^*$-fixed stable maps, the orbifold structures are determined by the Galois covers attached at these points. For a node, the orbifold structure is chosen by selecting a $$w\in \{0,...,r-1\}$$ such that the age of $V$ at this node is $$(age_{node}L +w)/r.$$ We substitute these ages into (\ref{eqn:GRR}) and write (\ref{eqn:GRR}) as 
\begin{equation}\label{eqn:GRR2}
\begin{split}
    ch(R^\bullet\pi_*f^*V)=&\pi_*(ch(f^*V)Td^\vee(L_{n+1}))\\
    &-\sum_{j=1}^{n(v)}\alpha_j\\
    &+\frac{1}{2}(\pi\circ\iota)_*r_{node}^2\beta_{node},
\end{split}    
\end{equation}
where 
\begin{equation*}
\begin{split}
&\alpha_j:=\sum_{m\geq 1}\frac{ev_j^*A_m}{m!}\psi_j^{m-1}\\
&\beta_{node}:=\sum_{m\geq 2}\frac{1}{m!}  ev_{node}^*A_m \frac{\psi_+^{m-1}+(-1)^m\psi_-^{m-1}}{\psi_++\psi_-},
\end{split}    
\end{equation*}
and $n(v)$ is the number of marked points at the vertex $v$. So 
\begin{equation}\label{eqn:GRR3}
\begin{split}
    ch_m(R^\bullet\pi_*f^*V)=&\pi_*(ch(f^*V)Td^\vee(L_{n+1}))_m\\
    &-\sum_{j=1}^{n(v)}(\alpha_j)_m\\
    &+\frac{1}{2}((\pi\circ\iota)_*r_{node}^2\beta_{node})_m. 
\end{split}    
\end{equation}

Using $$c(-E^\bullet)=\exp(\sum_{m\geq 1}(-1)^m (m-1)!ch_m(E^\bullet)),$$
we obtain a formula for $c(-R^\bullet\pi_*f^*V)\cap [\overline{\mathcal{M}}_{g(v),n(v)}(\sqrt[r]{L/\D}, \beta(v))]^{\on{vir}}$. Using that the pushforward via $\tau'$ has virtual degree $r^{2g-1}$ on genus $g$ stable map moduli, as calculated in \cite{TT}, we can get a formula for $\tau'_*(c(-R^\bullet\pi_*f^*V)\cap [\overline{\mathcal{M}}_{g(v),n(v)}(\sqrt[r]{L/\D}, \beta(v))]^{\on{vir}})$:

\begin{equation}\label{eqn:chern_class}
\begin{split}
\sum_{\overset{\Gamma\in G_{g,n,\beta}(\D)}{\chi\in \Gamma(\D), w\in W_{\Gamma, \chi, r}}}
\frac{r^{2g(v)-1-h^1(\Gamma)}}{|Aut(\Gamma)|}
&(j_{\Gamma, \chi})_*\left\{\prod_{\mathbf{v}\in V(\Gamma)}\exp(\sum_{m\geq 1}(-1)^m (m-1)!\pi_*(ch(f^*V)Td^\vee(L_{n+1}))_m)\right.\\
&\prod_{j=1}^{n(v)} \exp(\sum_{m\geq 1}(-1)^{m-1}(m-1)!(\alpha_j)_m)\\
&\left.\prod_{(h,h')\in E(\Gamma)} \frac{1-\exp(\sum_{m\geq 1}(-1)^m (m-1)! (\beta_{node})_m(\psi_h+\psi_{h'}))}{\psi_h+\psi_{h'}}\right\}\\
&\cap [\overline{\mathcal{M}}_{g(v),n(v)}(\D, \beta(v))]^{\on{vir}}.
\end{split}    
\end{equation}
Here the sum is over the set of $\D$-valued stable graphs denoted by $ G_{g,n,\beta}(\D)$ as in \cite{JPPZ2}; and $\chi\in \Gamma(\D)$ is a map that assigns to each half-edge a component of the inertia stack of $\D$, corresponding to assigning orbifold structures. Note that 
\begin{enumerate}
\item
For $(h, h')\in E(\Gamma)$, $\chi(h)$ and $\chi(h')$ are opposite. 
\item
For $\mathbf{v}\in V(\Gamma)$, we have $\int_{\beta{\mathbf{v}}}c_1(L)-\sum_{h\in H(\mathbf{v})} age_{\chi(h)}L \in \mathbb{Z}$. This is a consequence of Riemann-Roch for orbifold curves.
\end{enumerate}
We have used the equality $|E(\Gamma)|+\sum_{\mathbf{v}\in V(\Gamma)}(2g_\mathbf{v}-1)=2g(v)-1-h^1(\Gamma)$ for the prestable graph $\Gamma$ to get the factor $r^{2g(v)-1-h^1(\Gamma)}$ in the formula.

The map $$j_{\Gamma, \chi}:\overline{\mathcal{M}}_{\Gamma, \chi}\to \overline{\mathcal{M}}_{g(v), n(v)}(\D, \beta(v))$$ is the natural map from the component indexed by $\Gamma$ and $\chi$ into the moduli of stable maps to $\D$. 

Finally $W_{\Gamma, \chi, r}$ is the collection of $r$-twistings, which is the assignment $$h\mapsto w(h)\in \{0,...,r-1\},$$ such that 
\begin{enumerate}
    \item For $j\in L(\Gamma)$, we have $w(j)\equiv k_j-age_{\X_{i_j}}L\,\, mod\,\, r$, so the age of $V$ at marked point $j$ is $k_j/r$ for $k_j\geq 0$ or $1+k_j/r$ for $k_j<0$.
    
    \item For $(h, h')\in E(\Gamma)$, if $age_{\chi(h)}L=0$, then $w(h)+w(h')\equiv 0\,\, mod\,\, r$. If $age_{\chi(h)}L\neq 0$, then $w(h)+w(h')\equiv -1\,\, mod\,\, r$. These conditions ensure that $$(age_{\chi(h)}L+w(h))/r=1-(age_{\chi(h')}L+w(h'))/r.$$
    
    \item For $\mathbf{v}\in V(\Gamma)$, we have $\sum_{h\in H(\mathbf{v})}w(h)\equiv \int_{\beta{\mathbf{v}}}c_1(L)-\sum_{h\in H(\mathbf{v})} age_{\chi(h)}L \,\,mod \,\, r$. This follows from the lifting analysis of \cite{TT}.
\end{enumerate}

Fix $\Gamma$ and $\chi$ in (\ref{eqn:chern_class}). It follows from the description of $A_m$ that the summands in (\ref{eqn:chern_class}) are polynomials in $w\in W_{\Gamma, \chi, r}$. Pixton's polynomiality \cite[Appendix A]{JPPZ1} applies to show that $\tau'_*(c_i(-R^\bullet\pi_*f^*V)\cap [\overline{\mathcal{M}}_{g(v),n(v)}(\sqrt[r]{L/\D}, \beta(v))]^{\on{vir}})$ is a Laurent polynomial in $r$. Following  \cite[Proposition 5]{JPPZ1}, we can identify the lowest $r$ terms.
\begin{enumerate}
    \item 
After the summation over $r$-twistings, the lowest possible power of $r$ is $r^{h^1(\Gamma)-2i}$.

\item
The formula has a factor $r^{2g(v)-1-h^1(\Gamma)}$.
\item 
Finally there is a prefactor $r^{2i-2g(v)+1}$.
\end{enumerate}
Taken together, this shows that the lowest power of $r$ is $r^0$. This completes the proof.


\subsection{Proof of Theorem \ref{thm:poly-neg}}
The proof of Theorem \ref{thm:poly-neg} is similar to the proof of Theorem \ref{thm:poly}, but requires a more refined polynomiality than Proposition \ref{prop:poly}. 

Let $\overline{\M}_{g,\vec a}(\sqrt[r]{L/\D},\beta)$ be the moduli space of orbifold stable maps to $\sqrt[r]{L/\D}$, where $\vec a$ is a vector of ages. Let
\[
\pi: \mathcal C_{g,\vec a}(\sqrt[r]{L/\D},\beta)\rightarrow \overline{\M}_{g,\vec a}(\sqrt[r]{L/\D},\beta)
\]
 be the universal curve, 
\[
\mathcal L\rightarrow \mathcal C_{g,\vec a}(\sqrt[r]{L/\D},\beta)
\] 
is the universal $r$-th root. We consider the forgetful map 
\[
\tau^\prime: \overline{\M}_{g,\vec a}(\sqrt[r]{L/\D},\beta)\rightarrow \overline{\M}_{g,l(\vec a)}(\D,\beta)
\]
that forgets the $r$-th root construction.

\begin{prop}\label{prop-poly-neg}
For $r$ sufficiently large and $i\geq 0$, the class $$r^{i-g(v)+1}\tau^\prime_*(c_i(-R^\bullet\pi_*\mathcal{L})\cap [\overline{\M}_{g,\vec a}(\mathcal D_r,\beta)]^{\on{vir}})$$ is a polynomial in $r$ and it is constant in $r$ when $g(v)=0$, where $\tau^\prime$ is the map to the moduli space of stable maps to $\D$.
\end{prop}

The proof of Proposition \ref{prop-poly-neg} is similar to the proof in \cite[Appendix A]{FWY} and \cite[Section 4]{FWY19}. We briefly explain the idea here. First of all, in the proof of Theorem \ref{thm:poly}, we showed that, for sufficiently large $r$, the class $(\tau^\prime)_*\left[\bM_{g, \vec k,n}(\Y_{\D_0,r},\beta)\right]^{\on{vir}}$ is a polynomial in $r$ and it is constant in $r$ when $g=0$. The equivariant version of it is also true by considering equivariant theory as a limit of non-equivariant theory (see, for example \cite[Section 4.3]{FWY19}). Then the proposition follows from taking localization residue.

\begin{proof}[Proof of Proposition \ref{prop-poly-neg}]

Recall that the class $(\tau^\prime)_*\left[\bM_{g, \vec k,n}(\Y_{\D_0,r},\beta)\right]^{\on{vir}}$ is a polynomial in $r$ and it is constant in $r$ when $g=0$.
The first step is to prove it for families over a base.  Let $\pi: E\rightarrow B$ be a smooth morphism between two smooth algebraic varieties. Suppose that $E$ is also a $\mathbb C^*$-torsor over $B$. Let 
\[
\Y_{\D_0,r}\times _{\mathbb C^*}E=(\Y_{\D_0,r}\times E)/\mathbb C^*
\]
with $\mathbb C^*$ acts on both factors. We consider moduli space $\overline{\M}_{g,\vec k,n}(\Y_{\D_0,r}\times_{\mathbb C^*}E,\beta)$ of orbifold stable maps to $\Y_{\D_0,r}\times_{\mathbb C^*}E$, where the curve class $\beta$ is a fiber class (projects to $0$ on $B$). Let $$\left[\overline{\M}_{g,\vec k,n}(\Y_{\D_0,r}\times_{\mathbb C^*}E,\beta)\right]^{\on{vir}_{\pi}}$$ be the virtual cycle relative to the base $B$. Let
\[
\tau^\prime_{E}: \overline{\M}_{g,\vec k,n}(\Y_{\D_0,r}\times_{\mathbb C^*}E,\beta) \rightarrow \overline{\M}_{g,m+n}(\Y\times_{\mathbb C^*}E,\beta) 
\]
be the forgetful map that forgets the $r$-th root construction.
Then
\begin{align}\label{family-cycle}
\left(\tau^\prime_{E}\right)_*\left[\overline{\M}_{g,\vec k,n}(\Y_{\D_0,r}\times_{\mathbb C^*}E,\beta)\right]^{\on{vir}_{\pi}}
\end{align}
is a polynomial in $r$ and is constant in $r$ if $g=0$. The proof is parallel to the proof of Proposition \ref{prop:poly} as explained in \cite[Section 4.2]{FWY19}.

The next step is to prove that the equivariant cycle class
\begin{align}\label{equiv-cycle}
    \tau^\prime_*\left[\overline{\M}_{g,\vec k,n}(\Y_{\D_0,r},\beta)\right]^{\on{vir, eq}}
\end{align}
is a polynomial in $r$ and is constant in $r$ when $g=0$.
We follow the proof of \cite[Section 4.3]{FWY19}. The idea is that equivariant theory can be considered as a limit of non-equivariant theory. By \cite[Section 2.2]{EG}, the $i$-th Chow group of a space $X$ under an algebraic group $G$ can be defined as follows. Let $V$ be an $l$-dimensional representation of $G$ and $U\subset V$ be an equivariant open set where $G$ acts freely and whose complement has codimension more than $\dim X-i$. Then the $i$-th Chow group is defined as
\begin{align}\label{equiv-chow}
A_i^G(X)=A_{i+l-\dim G}((X\times U)/G).
\end{align}
To apply it to our case, we let $G=\mathbb C^*$ and $E=U=\mathbb C^N-\{0\}$, where $N$ is a sufficiently large integer. Then we have that $(X\times E)/\mathbb C^*$ is an $X$-fibration over $B=U/G=\mathbb P^{N-1}$. Note that
\[
\overline{\M}_{g,\vec k,n}(\Y_{\D_0,r}\times_{\mathbb C^*}E,\beta)\cong \left(\overline{\M}_{g,\vec k,n}(\Y_{\D_0,r},\beta)\times E\right)/\mathbb C^*
\]
as moduli spaces. For suitable $N$, (\ref{equiv-cycle}) identifies the equivariant Chow group with a non-equivariant model. Therefore, the equivariant cycle (\ref{equiv-cycle}) is identified with the non-equivariant cycle (\ref{family-cycle}) under (\ref{equiv-chow}). Therefore, the equivariant class (\ref{equiv-cycle}) is also a polynomial in $r$ and is constant in $r$ when $g=0$.

The last step is to consider localization residues of $\overline{\M}_{g,\vec k,n}(\Y_{\D_0,r},\beta)$. We consider the decorated graph with one vertex over $\mathcal D_0$ such that markings and edges are associated with the vector of ages $\vec a$. The localization residue is a polynomial in $r$ and is a constant when $g=0$. Then the cycle
\[
\tau^\prime_*\left(\sum_{i=0}^\infty\left(\frac tr\right)^{g-i-1}c_i(-R^\bullet\pi_*\mathcal L)\cap [\overline{\M}_{g,\vec a}(\mathcal D_r,\beta)]^{\on{vir}}\right),
\]
coming from the localization residue, is a polynomial in $r$ and is constant when $g=0$.
This is the conclusion of \cite[Theorem 4.1]{FWY19} for $\Y$ a smooth Deligne-Mumford stack. As a consequence (see also \cite[Corollary 4.2]{FWY19}), the cycle
\[
\tau^\prime_*\left((r)^{i-g+1}c_i(-R^\bullet\pi_*\mathcal L)\cap [\overline{\M}_{g,\vec a}(\mathcal D_r,\beta)]^{\on{vir}}\right)
\]
is a polynomial in $r$ and is constant when $g=0$. This concludes the proposition.
\end{proof}

\begin{proof}[Proof of Theorem \ref{thm:poly-neg}]
The proof is similar to the proof of Theorem \ref{thm:poly} with the help of Proposition \ref{prop-poly-neg}. The degeneration formula again reduces the proof to local models. The localization computation is similar to the computation in Section \ref{sec:localization} except that the $r$-dependence appears in the following form as the vertex contribution over $\D_0$:

\begin{align*}
&\left(\prod_{e\in E(v)}\frac{|G_{(e,v)}|}{r_{(e,v)}}\frac{r_{(e,v)}d_e}{t+\on{ev}_{e}^*c_1(L)-d_e\bar{\psi}_{(e,v)}}\right)\cdot\left(\sum_{i=0}^{\infty}(t/r)^{g(v)-1+|E(v)|-i+m_-(v)}c_i(-R^\bullet\pi_*\mathcal L)\right)\\
\notag =&\left(\prod_{e\in E(v)}\frac{|G_{(e,v)}'|}{1}\frac{d_e}{1+(\on{ev}_{e}^*c_1(L)-d_e\bar{\psi}_{(e,v)})/t}\right)\cdot\left(\sum_{i=0}^{\infty}t^{g(v)-i+m_-(v)-1}(r)^{i-g(v)+1-m_-(v)}c_i(-R^\bullet\pi_*\mathcal L)\right)\\
\notag=& r^{-m_-(v)}\left(\prod_{e\in E(v)}\frac{|G_{(e,v)}'|}{1}\frac{d_e}{1+(\on{ev}_{e}^*c_1(L)-d_e\bar{\psi}_{(e,v)})/t}\right)\cdot\left(\sum_{i=0}^{\infty}(t)^{g(v)-i+m_-(v)-1}(r)^{i-g(v)+1}c_i(-R^\bullet\pi_*\mathcal L)\right),
\end{align*}
where $m_-(v)$ is the number of large age markings attached to the vertex $v$ over $\D_0$. Multiplying by $r^{m_-}$,  then the polynomiality follows from Proposition \ref{prop-poly-neg}. This completes the proof of Theorem \ref{thm:poly-neg}.
\end{proof}

Theorem \ref{thm:poly-neg} implies that we can define relative Gromov-Witten invariants of an orbifold pair $(\X,\D)$ with negative contact orders as follows.
\begin{defn}
Let $\X$ be a smooth proper Deligne-Mumford stack over $\mathbb{C}$ with projective coarse moduli space. Let $\D\subset \X$ be a smooth irreducible divisor. The virtual cycle for the relative Gromov-Witten theory of the pair $(\X,\D)$  with negative contact orders is defined as follows:
\[
\left[\bM_{g,\vec k,n}(\X/\D,\beta)\right]^{\on{vir}}:=\left[r^{m_-}(\tau_{\on{orb}})_*\left[\bM_{g,\vec k,n}(\X_{\D,r},\beta)\right]^{\on{vir}}\right]_{r^0}\in A_*\left(\bM_{g,m+n}(\X,\beta)\times_{(\underline{\mathcal I}\X)^m} (\underline{\mathcal I}\D)^m\right).
\]
\end{defn}

\section{Gromov-Witten theory of multi-root stacks and its limit}\label{sec:application}
Let $X$ be a smooth projective variety\footnote{The main results of this paper also holds when $X$ is a smooth projective Deligne-Mumford stack. For simplicity, we only consider the case when $X$ is a smooth projective variety.} over $\mathbb{C}$ and let $$D_1, ..., D_n\subset X$$ be smooth irreducible divisors. Suppose $$D:=D_1+...+D_n$$ is simple normal crossing.

\begin{defn}\label{defn-multi-root}
For $\vec r=(r_1,\ldots, r_n)\in \mathbb{N}^n$, the multi-root stack $$X_{D, \vec{r}}:=X_{(D_1, r_1),...,(D_n, r_n)},$$ 
 is the stack whose objects over a scheme $S$ consist of the data 
$$f: S\to X, \{M_i: \text{ line bundle on }S\}, \{s_i\in H^0(M_i)\}, \{\phi_i:M_i^{\otimes r_i}\to f^*\mathcal{O}_X(D_i)\}$$
such that $s_i^{r_i}=\phi_i^*f^*\sigma_i$ for $i=1,...,n$.
\end{defn}

If $r_1,...,r_n$ are pairwise coprime, then $X_{D, \vec{r}}$ is smooth and has a well-defined Gromov-Witten theory.

For each $i=1, ..., n$, we can view $X_{D, \vec{r}}$ as $$(X_{(D_1, r_1),...,\widehat{(D_i, r_i)} ,..., (D_n, r_n)})_{(D_i, r_i)}.$$
Therefore Theorem \ref{thm:poly} applied to $X_{D, \vec{r}}$ implies polynomiality for each $r_i$, hence proves \cite[Conjecture 1.2]{TY20b}:

\begin{cor}\label{cor-limit}
For $r_1, ..., r_n$ sufficiently large, 
genus $0$ Gromov-Witten theory of $X_{D, \vec{r}}$, after multiplying by suitable powers of $r_i$, is independent of $r_1, ..., r_n$. Higher genus Gromov-Witten theory of $X_{D, \vec{r}}$, after multiplying by suitable powers of $r_i$, is a polynomial in $r_1, ..., r_n$.
\end{cor}

We may view the $r_1^0...r_n^0$ term of the Gromov-Witten theory of $X_{D, \vec{r}}$ as {\em formally} giving a  Gromov-Witten theory of infinite root stack $X_{D,\infty}$, which provides a virtual count of curves with tangency conditions along a simple normal crossing divisor. This can be viewed as analogous to logarithmic Gromov-Witten theory of the pair $(X,D)$. 

Now, we will state Corollary \ref{cor-limit} more precisely and define the formal Gromov-Witten theory of $X_{D,\infty}$. 
\begin{notation}
We will use ``relative marking'' and ``orbifold marking'' interchangeably. Terms like ``contact order'' and ``tangency condition'' will also be used. In Section \ref{sec:polynomiality}, we treat relative markings and interior markings separately. Here, it is more convenient to treat them all together. Therefore, the notation for the rest of the paper will be slightly different from the notation in Section \ref{sec:polynomiality}. We will use $n$ to denote the number of irreducible components of the divisor $D$ and use $m$ to denote the number of markings (including both relative and interior markings).  
\end{notation}

For any index set $I\subseteq\{1,\ldots, n\}$, we define 
\[
D_I:=\cap_{i\in I} D_i.
\]
Note that $D_I$ can be disconnected. In particular, we set 
\[
D_\emptyset:=X. 
\]

Let 
\[
\vec s=(s_1,\ldots,s_n)\in \mathbb Z^n.
\]
The vector $\vec s$ is used to record contact orders. Note that both positive and negative contact orders are allowed. We define
\[
I_{\vec s}:=\{i:s_i\neq 0\}\subseteq \{1,\ldots,n\}.
\]

Consider the vectors
\[
\vec s^j=(s_1^j,\ldots,s_n^j)\in (\mathbb Z)^n, \text{ for } j=1,2,\ldots,m,
\]
which satisfy the following condition:
\[
\sum_{j=1}^m s_i^j=\int_\beta[D_i], \text{ for } i\in\{1,\ldots,n\}.
\]
For sufficiently large\footnote{By sufficiently large $\vec r$, we mean $r_i$ are sufficiently large for all $i\in\{1,\ldots,n\}$.} $\vec r$, we consider the moduli space $$\bM_{g,\{\vec s^j\}_{j=1}^m}(X_{D,\vec r},\beta)$$ of genus $g$, degree $\beta\in H_2(X)$, $m$-pointed, orbifold stable maps to $X_{D,\vec r}$ with orbifold conditions specified by $\{\vec s^j\}_{j=1}^m$. Note that the $j$-th marking maps to twisted sector $D_{I_{\vec s^j}}$ with age
\[
\sum_{i: s^j_i>0} \frac{s^j_i}{r_i}+\sum_{i: s^j_i<0} \left(1+\frac{s^j_i}{r_i}\right).
\]
There are evaluation maps
\[
\on{ev}_j: \bM_{g,\{\vec s^j\}_{j=1}^m}(X_{D,\vec r},\beta) \rightarrow D_{I_{\vec s^j}}, \text{ for } j\in \{1,\ldots,m\}.
\]

Let \begin{itemize}
    \item $\gamma_j\in H^*(D_{I_{\vec s^j}})$, for $j\in\{1,2,\ldots,m\}$;
    \item $a_j\in \mathbb Z_{\geq 0}$, for $j\in \{1,2,\ldots,m\}$.
\end{itemize}
Gromov-Witten invariants of $X_{D,\vec r}$ are defined as follows
\begin{align*}
    \left\langle \gamma_1\bar{\psi}^{a_1},\ldots, \gamma_m\bar{\psi}^{a_m} \right\rangle_{g,\{\vec s^j\}_{j=1}^m,\beta}^{X_{D,\vec r}}:=
    \int_{\left[\bM_{g,\{\vec s^j\}_{j=1}^m}(X_{D,\vec r},\beta)\right]^{\on{vir}}}\on{ev}_1^*(\gamma_1)\bar{\psi}_1^{a_1}\cdots\on{ev}_m^*(\gamma_m)\bar{\psi}_m^{a_m}.
\end{align*}
We define
\[
s_{i,-}:=\#\{j: s_i^j<0\}, \text{ for } i=1,2,\ldots, n.
\]
Let 
\[
\tau:\bM_{g,\{\vec s^j\}_{j=1}^m}(X_{D,\vec r},\beta) \rightarrow \bM_{g,m}(X,\beta)\times_{X^m}\left(D_{I_{\vec s^1}}\times\cdots \times D_{I_{\vec s^m}}\right).
\]
be the forgetful map.

By Theorem \ref{thm:poly-neg}, the cycle class
\[
\left(\prod_{i=1}^n r_i^{s_{i,-}}\right)\tau_*\left(\left[\bM_{g,\{\vec s^j\}_{j=1}^m}(X_{D,\vec r},\beta)\right]^{\on{vir}}\right)
\]
is a polynomial in $r_i$ when $\vec r$ is sufficiently large.
We denote the constant term of the above polynomial as
\[
\left[\bM_{g,\{\vec s^j\}_{j=1}^m}(X_{D,\infty},\beta)\right]^{\on{vir}}:=\lim_{\vec r\rightarrow \infty}\left[\left(\prod_{i=1}^n r_i^{s_{i,-}}\right)\tau_*\left(\left[\bM_{g,\{\vec s^j\}_{j=1}^m}(X_{D,\vec r},\beta)\right]^{\on{vir}}\right)\right]_{\prod_{i=1}^n r_i^0}.
\]
It is considered as the virtual cycle of the formal Gromov-Witten theory of the infinite root stack $X_{D,\infty}$. 

Recall that there are evaluation maps
\[
\on{ev}_j: \bM_{g,\{\vec s^j\}_{j=1}^m}(X_{D,\vec r},\beta) \rightarrow D_{I_{\vec s^j}},
\]
for $j\in \{1,\ldots,m\}$.
We define the following evaluation maps
\[
\overline{\on{ev}}_j:\bM_{g,m}(X,\beta)\times_{X^m}\left(D_{I_{\vec s^1}}\times\cdots \times D_{I_{\vec s^m}}\right)\rightarrow D_{I_{\vec s^j}},
\]
such that
\[
\overline{\on{ev}}_j\circ \tau=\on{ev}_j,
\]
for $j\in \{1,\ldots,m\}$.

The formal Gromov-Witten invariants of $X_{D,\infty}$ can be defined as follows.
\begin{defn}\label{defn:GW_inv}
Let \begin{itemize}
    \item $\gamma_j\in H^*(D_{I_{\vec s^j}})$, for $j\in\{1,2,\ldots,m\}$;
    \item $a_j\in \mathbb Z_{\geq 0}$, for $j\in \{1,2,\ldots,m\}$.
\end{itemize}
The formal Gromov-Witten invariants of $X_{D,\infty}$ are defined as
\begin{align*}
    \left\langle [\gamma_1]_{\vec s^1}\bar{\psi}^{a_1},\ldots, [\gamma_m]_{\vec s^m}\bar{\psi}^{a_m} \right\rangle_{g,\{\vec s^j\}_{j=1}^m,\beta}^{X_{D,\infty}}:= \int_{\left[\bM_{g,\{\vec s^j\}_{j=1}^m}(X_{D,\infty},\beta)\right]^{\on{vir}}}\overline{\on{ev}}_1^*(\gamma_1)\bar{\psi}_1^{a_1}\cdots\overline{\on{ev}}_m^*(\gamma_m)\bar{\psi}_m^{a_m}.
\end{align*}
In other words,
\begin{align*}
    \left\langle [\gamma_1]_{\vec s^1}\bar{\psi}^{a_1},\ldots, [\gamma_m]_{\vec s^m}\bar{\psi}^{a_m} \right\rangle_{g,\{\vec s^j\}_{j=1}^m,\beta}^{X_{D,\infty}}:=\left[\left(\prod_{i=1}^n r_i^{s_{i,-}}\right)\left\langle \gamma_1\bar{\psi}^{a_1},\ldots, \gamma_m\bar{\psi}^{a_m} \right\rangle_{g,\{\vec s^j\}_{j=1}^m,\beta}^{X_{D,\vec r}}\right]_{\prod_{i=1}^n r_i^0}
\end{align*}
for sufficiently large $\vec r$. 
\end{defn}

Note that the $\bar{\psi}$-classes are pullback of $\psi$-classes on the moduli space $\bM_{g,m}(X,\beta)$ of stable maps to $X$.

\begin{rem}
When $D$ is irreducible, the formal Gromov-Witten theory of $X_{D,\infty}$ coincides with relative Gromov-Witten theory (possibly with negative contact orders) defined in \cite{FWY} and \cite{FWY19}. Relative Gromov-Witten theory in \cite{FWY} and \cite{FWY19} can also be defined using the usual relative Gromov-Witten theory of J. Li \cite{Li01}, \cite{Li02} and rubber theory of $D$. When $D$ is simple normal crossing, it is also possible to define the formal Gromov-Witten theory of $X_{D,\infty}$ in terms of the usual relative Gromov-Witten theory and rubber theory of $D_i$, but it will be more complicated and the combinatorics will be more involved than \cite{FWY} and \cite{FWY19}.  
\end{rem}

\section{Relative quantum cohomology}\label{sec:rel_QC}
In this section, we introduce quantum cohomology for $X_{D,\infty}$. We will call it \emph{relative quantum cohomology} of $(X,D)$ because we consider the formal Gromov-Witten theory of $X_{D,\infty}$ as a Gromov-Witten theory of $X$ relative to the simple normal crossing divisor $D$.

\subsection{The state space}\label{sec:state-space}
We briefly described the state space for the formal Gromov-Witten theory of infinite root stacks in \cite[Section 4]{TY20b}. In this section, we will provide more detailed discussion of it and its ring structure.

Following the description in \cite[Section 7.1]{FWY}, we formally define the state space for the Gromov-Witten theory of $X_{D,\infty}$ as the limit of the state space of $X_{D,\vec r}$:
\[
\mathfrak H:=\bigoplus_{\vec s\in \mathbb Z^n}\mathfrak H_{\vec s},
\]
where 
\[
\mathfrak H_{\vec s}:=H^*(D_{I_{\vec s}}).
\]
Note that 
\begin{itemize}
    \item 
$\mathfrak H_{\vec 0}:=H^*(D_\emptyset):=H^*(X)$;
\item if $\cap_{i: s_i\neq 0} D_i=\emptyset$, then $\mathfrak H_{\vec s}=0$.
\end{itemize} 
Each $\mathfrak H_{\vec s}$ naturally embeds into $\mathfrak H$. For an element $\gamma\in \mathfrak H_{\vec s}$, we write $[\gamma]_{\vec s}$ for its image in $\mathfrak H$. The pairing on $\mathfrak H$ 
\[
(-,-):\mathfrak H \times \mathfrak H\rightarrow \mathbb C
\]
is defined as follows: for $[\alpha]_{\vec s}$ and $[\beta]_{\vec s^\prime}$, define
\begin{equation}\label{eqn:pairing}
\begin{split}
([\alpha]_{\vec s},[\beta]_{\vec s^\prime}) = 
\begin{cases}
\int_{D_{I_{\vec s}}} \alpha\cup\beta, &\text{if } \vec s=-\vec s^\prime;\\
0, &\text{otherwise. }
\end{cases}
\end{split}
\end{equation}
The pairing on the rest of the classes is generated by linearity. Recall that $D_\emptyset=X$, therefore
\[
([\alpha]_{\vec s},[\beta]_{\vec s^\prime}) =\int_X \alpha\cup\beta, \text{ if } \vec s=-\vec s^\prime=\vec 0.
\]
We choose a basis $\{T_{I,k}\}_k$ for $H^*(D_I)$. When $I=\emptyset$, we can also simply write $\{T_k\}_k$ for a basis for $H^*(X)$. Then we can define a basis of $\mathfrak H$ as follows:
\[
\tilde{T}_{\vec s,k}=[T_{I_{\vec s},k}]_{\vec s}.
\]
Let $\{T_{I}^k\}$ be the dual basis of $\{T_{I,k}\}$ under the Poincar\'e pairing of $H^*( D_I)$. Define
\[
\tilde{T}_{\vec s}^k=[T_{I_{\vec s}}^k]_{\vec s}.
\]
Then $\{\tilde T_{\vec s}^k\}$ form a dual basis of $\{\tilde T_{\vec s, k}\}$ under the pairing of $\mathfrak H$. Note that the dual of $\tilde T_{\vec s,k}$ is $\tilde T_{-\vec s}^k$ under the pairing of $\mathfrak H$.

\begin{defn}
For $[\alpha],[\beta]\in \mathfrak H$, the product $[\alpha]\cdot[\beta]$ is defined as follows: for $[\gamma]\in \mathfrak H$,
\[
([\alpha]\cdot[\beta],[\gamma]):=\langle [\alpha],[\beta],[\gamma]\rangle_{0,3,0}^{X_{D,\infty}},
\]
where the right-hand side is the genus zero, degree zero invariant of $X_{D,\infty}$ with three marked points.
\end{defn}

Similar to \cite{FWY}, the product structure can be written down explicitly, by computing the genus zero, degree zero $3$-pointed invariants.

Note that the ring $\mathfrak H$ is multi-graded. There are gradings with respect to contact orders $\vec s$:
\begin{align}\label{deg-i}
\deg^{i}([\alpha]_{\vec s})=s_i.
\end{align}
There is one grading for the cohomological degree of the class. Suppose $\alpha\in \mathfrak H_{\vec s}$ is a cohomology class of real degree $d$. Then we define,
\begin{align}\label{deg-0}
\deg^{0}([\alpha]_{\vec s})=d/2+\#\{i:s_i<0\}.
\end{align}
Note that there is a shift of the degree in (\ref{deg-0}). It already appears in \cite[Section 7.1]{FWY} when $D$ is irreducible. One can simply think about the degree (\ref{deg-0}) as a limit of the orbifold degree (shifted by ages).

Let $[\gamma_j]_{\vec s^j}\in \mathfrak H$ and $a_j\in \mathbb Z_{\geq 0}$, for $j\in\{1,\ldots,m\}$, where
\[
\vec s^j=(s_1^j,\ldots,s_n^j)\in (\mathbb Z)^n.
\]
Recall that the formal Gromov-Witten invariant of $X_{D,\infty}$ is denoted by
\begin{align}\label{inv-X-D}
\left\langle [\gamma_1]_{\vec s^1}\bar{\psi}^{a_1},\ldots, [\gamma_m]_{\vec s^m}\bar{\psi}^{a_m} \right\rangle_{g,\{\vec s^j\}_{j=1}^m,\beta}^{X_{D,\infty}}.
\end{align}
The invariant (\ref{inv-X-D}) is zero unless it satisfies the virtual dimension constraint
\begin{align}\label{vir-dim}
    (1-g)(\dim_{\mathbb C}X-3)+m+\int_\beta c_1(T_X)-\int_\beta [D]=\sum_{j=1}^m \deg^{0}([\gamma_j]_{\vec s^j})+\sum_{j=1}^m a_j.
\end{align}
We will also denote the invariant (\ref{inv-X-D}) by $\langle \cdots\rangle_{g,m,\beta}^{X_{D,\infty}}$ if the contact order information is clear from the insertion. Sometimes, we will abbreviate it to $\langle \cdots\rangle$ for simplicity.

\subsection{Universal equations}\label{Sec:univ_eqns}

Absolute Gromov-Witten invariants are known to satisfy the following universal equations: string equation, divisor equation, dilaton equation, topological recursion relation (TRR), and Witten-Dijkgraaf-Verlinde-Verlinde (WDVV) equation (see, for example, \cite{AGV}, \cite{Tseng} for universal equations for orbifold Gromov-Witten invariants). It was proved in \cite{FWY} that relative Gromov-Witten invariants also satisfy these universal equations. Our definition of the formal Gromov-Witten invariants of infinite root stacks is taken as the limit of orbifold Gromov-Witten invariants of finite root stacks. It is straightforward to show that these universal equations  are preserved under the limit. Therefore, we have the following universal equations for the formal Gromov-Witten invariants of infinite root stacks.

Let $\vec s^0=\vec 0$, we have
\begin{prop}[String equation]
\begin{align}\label{string-equ}
    &\left\langle [1]_{\vec 0},[\gamma_1]_{\vec s^1}\bar{\psi}^{a_1},\ldots, [\gamma_m]_{\vec s^m}\bar{\psi}^{a_m} \right\rangle_{g,\{\vec s^j\}_{j=0}^m,\beta}^{X_{D,\infty}}\\
   \notag =&\sum_{j=1}^m \left\langle [\gamma_1]_{\vec s^1}\bar{\psi}^{a_1},\ldots,[\gamma_j]_{\vec s^j}\bar{\psi}^{a_j-1},\ldots, [\gamma_m]_{\vec s^m}\bar{\psi}^{a_m} \right\rangle_{g,\{\vec s^j\}_{j=1}^m,\beta}^{X_{D,\infty}}.
\end{align}

\end{prop}

\begin{prop}[Divisor equation]
For $\gamma\in H^2(X)$,
\begin{align*}
    &\left\langle [\gamma]_{\vec 0},[\gamma_1]_{\vec s^1}\bar{\psi}^{a_1},\ldots, [\gamma_m]_{\vec s^m}\bar{\psi}^{a_m} \right\rangle_{g,\{\vec s^j\}_{j=0}^m,\beta}^{X_{D,\infty}}=\left(\int_\beta \gamma\right)\left\langle [\gamma_1]_{\vec s^1}\bar{\psi}^{a_1},\ldots, [\gamma_m]_{\vec s^m}\bar{\psi}^{a_m} \right\rangle_{g,\{\vec s^j\}_{j=1}^m,\beta}^{X_{D,\infty}}\\
    &\quad +\sum_{j=1}^m \left\langle [\gamma_1]_{\vec s^1}\bar{\psi}^{a_1},\ldots,[\gamma_j\cdot \gamma]_{\vec s^j}\bar{\psi}^{a_j-1},\ldots, [\gamma_m]_{\vec s^m}\bar{\psi}^{a_m} \right\rangle_{g,\{\vec s^j\}_{j=1}^m,\beta}^{X_{D,\infty}}.
\end{align*}
\end{prop}

\begin{prop}[Dilaton equation]
\begin{align*}
    &\left\langle \bar{\psi}[1]_{\vec 0},[\gamma_1]_{\vec s^1}\bar{\psi}^{a_1},\ldots, [\gamma_m]_{\vec s^m}\bar{\psi}^{a_m} \right\rangle_{g,\{\vec s^j\}_{j=0}^m,\beta}^{X_{D,\infty}}
    =& (2g-2+m)\left\langle [\gamma_1]_{\vec s^1}\bar{\psi}^{a_1},\ldots, [\gamma_m]_{\vec s^m}\bar{\psi}^{a_m} \right\rangle_{g,\{\vec s^j\}_{j=1}^m,\beta}^{X_{D,\infty}}.
\end{align*}
\end{prop}

\begin{prop}[TRR]
In genus zero,
\begin{align}\label{trr}
&\left\langle [\gamma_1]_{\vec s^1}\bar{\psi}^{a_1+1},\ldots, [\gamma_m]_{\vec s^m}\bar{\psi}^{a_m} \right\rangle_{0,\{\vec s^j\}_{j=1}^m,\beta}^{X_{D,\infty}}\\
\notag=&\sum \left\langle [\gamma_1]_{\vec s^1}\bar{\psi}^{a_1},\prod_{j\in S_1}[\gamma_j]_{\vec s^j}\bar{\psi}^{a_j}, \tilde{T}_{\vec s,k} \right\rangle_{0,\{\vec s^j\}_{j\in S_1\cup\{1\}},\vec s,\beta_1}^{X_{D,\infty}}\\
\notag& \qquad \cdot\left\langle \tilde{T}_{-\vec s}^k,[\gamma_2]_{\vec s^2}\bar{\psi}^{a_2},[\gamma_3]_{\vec s^3}\bar{\psi}^{a_3},\prod_{j\in S_2}[\gamma_j]_{\vec s^j}\bar{\psi}^{a_j} \right\rangle_{0,-\vec s,\{\vec s^j\}_{j\in S_2\cup\{2,3\}},\beta_2}^{X_{D,\infty}},
\end{align}
where the sum is over all splittings of $\beta_1+\beta_2=\beta$, all indices $\vec s, k$ of basis, and all splittings of disjoint sets $S_1$, $S_2$ with $S_1\cup S_2=\{4,\ldots,m\}$. Note that the right-hand side is a finite sum.
\end{prop}

\begin{prop}[WDVV] In genus zero,
\begin{align}\label{wdvv}
    &\sum \left\langle [\gamma_1]_{\vec s^1}\bar{\psi}^{a_1},[\gamma_2]_{\vec s^2}\bar{\psi}^{a_2},\prod_{j\in S_1}[\gamma_j]_{\vec s^j}\bar{\psi}^{a_j}, \tilde{T}_{\vec s,k} \right\rangle_{0,\{\vec s^j\}_{j\in S_1\cup\{1,2\}},\vec s,\beta_1}^{X_{D,\infty}}\\
\notag & \qquad \cdot\left\langle \tilde{T}_{-\vec s}^k,[\gamma_3]_{\vec s^3}\bar{\psi}^{a_3},[\gamma_4]_{\vec s^4}\bar{\psi}^{a_4},\prod_{j\in S_2}[\gamma_j]_{\vec s^j}\bar{\psi}^{a_j} \right\rangle_{0,-\vec s,\{\vec s^j\}_{j\in S_2\cup\{3,4\}},\beta_2}^{X_{D,\infty}}\\
\notag=&\sum \left\langle [\gamma_1]_{\vec s^1}\bar{\psi}^{a_1},[\gamma_3]_{\vec s^3}\bar{\psi}^{a_3}\prod_{j\in S_1}[\gamma_j]_{\vec s^j}\bar{\psi}^{a_j}, \tilde{T}_{\vec s,k} \right\rangle_{0,\{\vec s^j\}_{j\in S_1\cup\{1,3\}},\vec s,\beta_1}^{X_{D,\infty}}\\
\notag & \qquad \cdot\left\langle \tilde{T}_{-\vec s}^k,[\gamma_2]_{\vec s^2}\bar{\psi}^{a_2},[\gamma_4]_{\vec s^4}\bar{\psi}^{a_4},\prod_{j\in S_2}[\gamma_j]_{\vec s^j}\bar{\psi}^{a_j} \right\rangle_{0,-\vec s,\{\vec s^j\}_{j\in S_2\cup\{2,4\}},\beta_2}^{X_{D,\infty}},
\end{align}
where each sum is over all splittings of $\beta_1+\beta_2=\beta$, all indices $\vec s, k$ of basis, and all splittings of disjoint sets $S_1$, $S_2$ with $S_1\cup S_2=\{5,\ldots,m\}$. Note that both sides are finite sums.
\end{prop}

\begin{rem}
Just like the WDVV equation for absolute Gromov-Witten theory implies the associativity of the quantum cohomology, the WDVV equation for the formal Gromov-Witten theory of infinite root stacks also implies the associativity of the relative quantum cohomology. 
Note that in \cite{GS19}, it requires extensive arguments to prove the associativity for (the degree zero part of) the relative quantum cohomology. While in our case, we obtain the associativity for free. Since we do not know the relation between the invariants that we considered here and the punctured invariants in \cite{GS19} and \cite{ACGS}, it is not known that if our approach will provide an easier proof of the associativity in \cite{GS19}. 

The compatibility between this new theory and the Gross-Siebert program will be discussed in Section \ref{sec:frob}.
\end{rem}

\subsection{Relative quantum cohomology ring}\label{sec:rel-qh}

Let $ t=\sum t_{\vec s,k}\tilde{T}_{\vec s,k}$ where $t_{\vec s,k}$ are formal variables. Let $\mathbb C[\![\on{NE}(X)]\!]$ be the Novikov ring, where $q$ is the Novikov variable and $\on{NE}(X)$ be the cone of effective curve classes in $X$. We denote the formal power series ring with variables $t_{\vec s,k}$ by
\[
\mathbb C[\![\on{NE}(X)]\!][\![\{t_{\vec s,k}\}]\!].
\]
Note that there are infinitely many variables. We will work on a completion of this ring. Consider the ideals
\[
I_p=(\{t_{\vec s,k}\}_{|s_i|\geq p, \forall i})
\]
for $p\geq 0$. These ideals form a chain
\[
I_0\supset I_1 \supset I_2 \supset \cdots.
\]
Now we have the completion
\[
\mathbb C[\![\on{NE}(X)]\!]\widehat{[\![\{t_{\vec s,k}\}]\!]}=\varprojlim \mathbb C[\![\on{NE}(X)]\!][\![\{t_{\vec s,k}\}]\!]/I_p.
\]

The \emph{genus-zero potential for the Gromov-Witten theory of infinite root stacks} is defined to be
\[\Phi_0( t)=\sum_{m\geq 3}\sum_{\beta}\frac{1}{m!}\langle t,\cdots,t\rangle_{0,m,\beta}^{X_{D,\infty}} q^{\beta}\in\mathbb C[\![\on{NE}(X)]\!]\widehat{[\![\{t_{\vec s,k}\}]\!]}.\]
Note that $\Phi_0$ is a formal function in variables $\{t_{\vec s,k}\}$. To define a ring structure on $\mathbb C[\![\on{NE}(X)]\!]\widehat{[\![\{t_{\vec s,k}\}]\!]}$,  we define the quantum product $\star$ by the following
\[
\tilde{T}_{\vec s^1,k_1}\star \tilde{T}_{\vec s^2,k_2}=\sum_{\vec s^3,k_3}\frac{\partial^3 \Phi_0}{\partial t_{\vec s^1,k_1}\partial t_{\vec s^2,k_2}\partial t_{\vec s^3,k_3}}\tilde{T}_{-\vec s^3}^{k_3}.
\]
Recall that $\tilde{T}_{\vec s^3,k_3}$ and $\tilde{T}_{-\vec s^3}^{k_3}$ are dual to each other under the pairing.

One can also define small relative quantum cohomology ring by setting $t_{\vec s,k}=0$ if $\vec s\neq \vec 0$ or $\tilde{T}_{\vec 0,k}\not\in H^0(X)\oplus H^2(X)\subset \mathfrak H_{\vec 0}$ in the formal function
\[
\frac{\partial^3 \Phi_0}{\partial t_{\vec s^1,k_1}\partial t_{\vec s^2,k_2}\partial t_{\vec s^3,k_3}}.
\]
The small relative quantum product is denoted by $\star_{\on{sm}}$. The small relative quantum cohomology ring is denoted by $QH(X_{D,\infty})$.

Similar to the absolute Gromov-Witten theory, under the specialization $q=0$ and $ t=0$, we obtain the product structure of the state space in Section \ref{sec:state-space}:
\[
\tilde{T}_{\vec s^1,k_1}\star_{q=0, t=0} \tilde{T}_{\vec s^2,k_2}=\sum_{\vec s^3,k_3}\left\langle \tilde{T}_{\vec s^1,k_1},\tilde{T}_{\vec s^2,k_2}, \tilde{T}_{\vec s^3,k_3}\right\rangle_{0,3,0}^{X_{D,\infty}}\tilde{T}_{-\vec s^3}^{k_3}.
\]

Relative quantum cohomology ring is a multi-graded ring. Similar to \cite[Section 7.3]{FWY}, the gradings are defined as extensions of $\deg^{i}$ in (\ref{deg-0}) and (\ref{deg-i}). Furthermore, we define
\[
\deg^{(i)}(q^\beta)=\int_\beta D_i, \quad \deg^{(i)}(t_{\vec s,k})=-s_i, \text{ for }  i\in\{1,\ldots, n\},
\]
\[
\deg^{(0)}(q^\beta)=\int_\beta c_1(T_X(-\log D)), \quad \deg^{(0)}(t_{\vec s,k})=1-\deg^{(0)}(\tilde{T}_{\vec s,k}).
\]

\section{Givental formalism}\label{sec:givental}
In this section, we set up Givental formalism for genus zero formal Gromov-Witten theory of the infinite root stack $X_{D,\infty}$ following \cite{Givental04}. A mirror theorem for infinite root stacks has already been proved in \cite{TY20b}. This section provides the necessary foundation for \cite{TY20b}.

Consider the space
\[
\mathcal H=\mathfrak H \otimes_{\mathbb C}\mathbb C[\![\on{NE}(X)]\!](\!(z^{-1})\!),
\]
where $(\!(z^{-1})\!)$ means formal Laurent series in $z^{-1}$.

There is a $\mathbb C[\![\on{NE}(X)]\!]$-valued symplectic form
\[
\Omega(f,g)=\text{Res}_{z=0}(f(-z),g(z))dz, \text{ for } f,g\in \mathcal H,
\]
where the pairing $(f(-z),g(z))$ takes values in $ \mathbb C[\![\on{NE}(X)]\!](\!(z^{-1})\!)$ and is induced by the pairing on $\mathfrak H$.

Consider the following polarization
\[
\mathcal H=\mathcal H_+\oplus\mathcal H_-,
\]
where
\[
\mathcal H_+=\mathfrak H \otimes_{\mathbb C} \mathbb C[\![\on{NE}(X)]\!][z], \quad \text{and} \quad \mathcal H_-=z^{-1}\mathfrak H \otimes_{\mathbb C} \mathbb C[\![\on{NE}(X)]\!][\![z^{-1}]\!].
\]
There is a natural symplectic identification between $\mathcal H_+\oplus \mathcal H_-$ and the cotangent bundle $T^*\mathcal H_+$.

For $l\geq 0$, we write $ t_l=\sum\limits_{\vec s,k} t_{l;\vec s,k}\widetilde T_{\vec s,k}$ where $t_{l;\vec s,k}$ are formal variables. Also write
\[
\mathbf t(z)=\sum\limits_{l=0}^\infty t_l z^l.
\]

\emph{The genus $g$ descendant Gromov-Witten potential of $X_{D,\infty}$} is defined as
\[
\mathcal F^g_{X_{D,\infty}}(\bt(z))=\sum\limits_\beta \sum\limits_{m=0}^\infty \dfrac{q^\beta}{m!} \left\langle\mathbf t(\bar\psi),\ldots,\mathbf t(\bar\psi)\right\rangle_{g,m,\beta}^{X_{D,\infty}}.
\]
\emph{The total descendant Gromov-Witten potential} is defined as
\[
\mathcal D_{X_{D,\infty}}(\mathbf t):=\exp \left(\sum_{g \geq 0}\hbar^{g-1}\mathcal F_{X_{D,\infty}}^g(\mathbf t)\right).
\]

Following \cite{Givental04}, we define the dilaton-shifted coordinates of $\mathcal H_+$
\[
\mathbf q(z)=q_0+q_1z+q_2z^2+\ldots=-z+t_0+t_1z+t_2z^2+\ldots.
\]
\[
\mathbf p(z)=p_0z^{-1}+p_1z^{-2}+\ldots = \sum\limits_{l\leq -1} \sum\limits_{\vec s,k} p_{l;\vec s,k}\tilde T_{-\vec s}^k z^l.
\]
Coordinates $\mathbf p(z)$ in $\mathcal H_-$ are chosen so that $q, p$ form Darboux coordinates. 

One can consider the graph of the differential $d\mathcal F^0_{X_{D,\infty}}$:
\[
\mathcal L_{X_{D,\infty}}:=\{(\mathbf p,\mathbf q)| \mathbf p=d_{\mathbf q}\mathcal F^0_{X_{D,\infty}}\}\subset \mathcal H=T^*\mathcal H_+.
\]
Equivalently, a (formal) point in $\mathcal L_{X_{D,\infty}}$ can be explicitly written as
\[
-z+\mathbf t(z)+\sum\limits_{\beta} \sum\limits_{m} \sum\limits_{\vec s,k} \dfrac{q^\beta}{m!} \left\langle\dfrac{\tilde T_{\vec s,k}}{-z-\bar\psi},\mathbf t(\bar\psi),\ldots,\mathbf t(\bar\psi)\right\rangle_{0,m+1,\beta}^{X_{D,\infty}} \tilde T_{-\vec s}^k.
\]

By \cite[Theorem 1]{Givental04} (see also \cite[Theorem 3.1.1]{Tseng} for orbifold Gromov-Witten theory), string equation, dilaton equation and topological recursion relations imply the following property.
\begin{prop}\label{prop-tangent}
$\mathcal L_{X_{D,\infty}}$ is the formal germ of a Lagrangian cone with vertex at the origin such that each tangent space $T$ to the cone is tangent to the cone exactly along $zT$.
\end{prop}

Following \cite{Barannikov}, the set of tangent spaces $T$ to the cone $\mathcal L$ satisfying Proposition \ref{prop-tangent} carries a canonical Frobenius structure. We refer to \cite{Givental04} for more details. 

\begin{defn}
We define the $J$-function $J_{X_{D,\infty}}(t,z)$ as follows,
\[
J_{X_{D,\infty}}(t,z)=z+t+\sum_{m\geq 1, \beta\in \on{NE}(X)}\sum_{\vec s,k}\frac{q^\beta}{m!}\left\langle\dfrac{\tilde T_{\vec s,k}}{-z-\bar\psi},t,\ldots, t\right\rangle_{0,m+1,\beta}^{X_{D,\infty}} \tilde T_{-\vec s}^k.
\]
\end{defn}
The $J$-function is a formal power series in coordinates $t_{\vec s, k}$ of $t=\sum t_{\vec s,k}\tilde{T}_{\vec s,k}\in \mathfrak H$ taking values in $\mathcal H$. The point of $\mathcal L_{X_{D,\infty}}$ above $-z+t\in \mathcal H_+$ is $J_{\mathcal X_{D,\infty}}(t,-z)$. In other words, $J_{\mathcal X_{D,\infty}}(t,-z)$ is the intersection of $\mathcal L_{X_{D,\infty}}$ with $(-z+t)+\mathcal H_-$.

The $I$-function $I_{X_{D,\infty}}$ for $X_{D,\infty}$ is constructed in \cite[Section 4]{TY20b} as a hypergeometric modification of the $J$-function of $X$. Using Givental formalism that we just developed, a mirror theorem for the infinite root stack $X_{D,\infty}$ can be stated as follows.
\begin{thm}\label{thm:mirror}
Let $X$ be a smooth projective variety. Let $D:=D_1+D_2+...+D_n$ be a simple normal-crossing divisor with $D_i\subset X$ smooth, irreducible and nef. The $I$-function $I_{X_{D,\infty}}$, defined in \cite[Section 4]{TY20b}, of the infinite root stack $X_{D,\infty}$ lies in Givental's Lagrangian cone $\mathcal L_{X_{D,\infty}}$ of $X_{D,\infty}$.
\end{thm}

\begin{rem}
The $I$-function $I_{D,\infty}$ considered in \cite[Section 4]{TY20b} is taken as a limit of the $I$-functions for finite root stacks. Theorem \ref{thm:mirror} holds for both non-extended $I$-function and extended $I$-function. When $D$ is a smooth divisor, Theorem \ref{thm:mirror} is simply \cite[Theorem 1.4]{FTY} for non-extended $I$-function and \cite[Theorem 1.5]{FTY} for extended $I$-function of the smooth pair $(X,D)$.
\end{rem}

\section{Virasoro constraints}\label{Sec:vir}
Givental formalism implies Virasoro constraints for genus zero Gromov-Witten invariants of infinite root stacks. We briefly describe it in this section.

Given a class $[\alpha]_{\vec s}\in \mathfrak H$ such that $\alpha\in H^{p,q}(D_{I_{\vec s}})$. Note that when $\vec s=\vec 0$, we use the convention that $D_{I_{\vec 0}}=D_\emptyset=X$. We define two operators $\rho,\mu$ as follows.
\[
\rho([\alpha]_{\vec s})=\left[\alpha\cup c_1(T_X(-\log D))\right]_{\vec s},
\]
\[
\mu([\alpha]_{\vec s})=\left[(\dim_{\mathbb C}(X)/2-p-\#\{i:s_i<0\})\alpha\right]_{\vec s}.
\]

Then we define the following transformations:
\begin{align*}
    \begin{split}
        &l_{-1}=z^{-1},\\
        &l_{0}=zd/dz + 1/2 + \mu + \rho/z,\\
        &l_m=l_0(zl_0)^m, \quad m\geq 1.
    \end{split}
\end{align*}
Recall that an operator $A:\mathcal H\rightarrow \mathcal H$ is called infinitesimal symplectic if it satisfies
\[
\Omega(A(f),g)+\Omega(f,A(g))=0 \text{ for all } f,g\in \mathcal H.
\]
One can check that $l_{m}$ are infinitesimal symplectic. Furthermore, the operator $l_m$ satisfies the following commutation relations:
\[
\{l_m,l_n\}=(n-m)l_{m+n},
\]
where 
$\{-,-\}$ is the Poisson bracket.

Following \cite{Givental04}, an infinitesimal symplectic transformation $A$ gives rise to a vector field on $\mathcal H$ in the following way. The tangent space of $\mathcal H$ at a point $f\in \mathcal H$ can be naturally identified with $\mathcal H$ itself. One obtains a tangent vector field on $\mathcal H$ by assigning the vector $A(f)\in T_f \mathcal H$ to the point $f$. The following proposition follows from \cite[Theorem 6]{Givental04}.
\begin{prop}\label{prop-oper}
The vector fields defined by the operators $l_m$, $m=1,2,\ldots,$ are tangent to the Lagrangian cone $\mathcal L$.
\end{prop}
Therefore, $l_m$ are associated with Hamitonian functions on $\mathcal L$:
\[
f\mapsto \frac 12 \Omega(l_mf,f).
\]
We define the quantization of the quadratic monomials using the following standard rules:
\begin{align*}
    (q_{l;\vec s,k}q_{l';\vec s',k'})^{\wedge}&=q_{l;\vec s,k}q_{l';\vec s',k'}/\hbar, \\ 
    (q_{l;\vec s,k}p_{l';\vec s',k'})^{\wedge}&=q_{l;\vec s,k}\partial/\partial q_{l';\vec s',k'}, \\ 
    (p_{l;\vec s,k}p_{l';\vec s',k'})^{\wedge}&=\hbar \partial^2/\partial q_{l;\vec s,k}\partial q_{l';\vec s',k'}.
\end{align*}
Hence, we obtain a sequence of quantized operators
\[
L_m=\hat{l}_m.
\]

Then the following genus zero Virasoro constraints follow from the fact that $l_m$ is infinitesimal symplectic and Proposition \ref{prop-oper}.
\begin{prop}
For $m\geq -1$, we have the following identity
\[
\left[e^{-\mathcal F^0(\mathbf t)/\hbar} L_m e^{\mathcal F^0(\mathbf t)/\hbar}\right]_{\hbar^{-1}}=0,
\]
where $[\cdots]_{\hbar^{-1}}$ means taking the $\hbar^{-1}$-coefficient.
\end{prop}

\section{Intrinsic mirror symmetry}\label{sec:frob}

In this section, we apply invariants of $X_{D,\infty}$ and relative quantum cohomology $QH(X_{D,\infty})$ to study the intrinsic mirror symmetry of the Gross-Siebert program in our setting.

The Frobenius structure conjecture for log pairs $(X,D)$ was stated in the first arXiv version of \cite{GHK}. The Frobenius structure conjecture predicts that there is a commutative associative algebra associated to the pair $(X,D)$ and the spectrum of the algebra is mirror to $(X,D)$. The conjecture was proved in \cite{GS19} by explicitly defining all structure constants in terms of punctured Gromov-Witten invariants. It was proved for cluster log pairs in \cite{Mandel19} and for affine log Calabi-Yau varieties containing a torus in \cite{KY}. Our construction will also provide a commutative associative algebra associated to log pairs $(X,D)$ when $D$ is a simple normal crossing divisor.
We briefly review the conjecture and explain how our construction can be used to study the conjecture as well as the mirror construction in the Gross-Siebert program \cite{GS18} and \cite{GS19} in our setting. 

Let $D=D_1+\cdots+D_n$ and $S$ be the dual intersection complex of $D$. That is, $S$ is the simplicial complex with vertices $v_1,\ldots, v_n$ and simplices $\langle v_{i_1},\ldots, v_{i_p}\rangle$ corresponding to non-empty intersections $D_{i_1}\cap \cdots \cap D_{i_p}$. Let $B$ denote the cone over $S$ and $\Sigma$  be the induced
simplicial fan in $B$. Let $B(\mathbb Z)$ be the set of integer points of $B$. Let $QH_{\log}^0(X,D)$ be the degree $0$ subalgebra of the relative quantum cohomology ring $QH_{\log}^*(X,D)$. There is a bijection between points $p\in B(\mathbb Z)$ and prime fundamental classes $\vartheta_p\in QH_{\log}^0(X,D)$. 

Suppose we are given points $p_1,\ldots, p_m\in B_0(\mathbb Z)$, where $B_0=B\setminus \{0\}$. Each $p_i$ can be written as a linear combination of primitive generators $v_{ij}$ of rays in $\Sigma$:
\[
p_i=\sum_{j} m_{ij}v_{ij},
\]
where the ray generated by $v_{ij}$ corresponds to a divisor $D_{ij}$. 

We assume $(K_X+D)$ is nef or anti-nef. For $m\geq 2$, using the result of \cite{GS13} and \cite{AC14}, one can define the associated log Gromov-Witten invariant
\begin{align}\label{log-inv}
N^\beta_{p_1,\ldots, p_m,0}:=\int_{[\bM_{0,m+1}(X/D,\beta)]^{\on{vir}}}\on{ev}_0^*[pt]\cdot \psi_0^{m-2},
\end{align}
where $\bM_{0,m+1}(X/D,\beta)$ is the moduli stack of logarithmic stable maps which provides a compactification for the space of stable maps
\[
f:(C,x_0,x_1,\ldots,x_m)\rightarrow X
\]
such that $f_*[C]=\beta$, and $C$ meets $D_{ij}$ at $x_i$ with contact order $m_{ij}$ for each $i,j$ and contact order zero with $D$ at $x_0$. Note that no punctured invariants are involved at this point. 

The Frobenius structure conjecture can be partially rephrased as

\begin{conjecture}\label{conj-frob}
The coefficient of $\vartheta_0$ in the product $\vartheta_{p_1}\star\cdots \star\vartheta_{p_m}$ is 
\[
\sum_{\beta\in H_2(X)} N^\beta_{p_1,\ldots, p_m,0}q^\beta.
\]
\end{conjecture}

Conjecture \ref{conj-frob} will be rephrased in our language in the following sections.

\subsection{The mirror algebra}
Let $QH^0(X_{D,\infty})$ be the degree zero part of the relative quantum cohomology ring $QH(X_{D,\infty})$ in Section \ref{sec:rel-qh}. The degree zero part means the degree in (\ref{deg-0}) is zero. For a cohomology class $[\alpha]_{\vec s}\in \mathfrak H_{\vec s}$ of real degree $d$ to be of degree zero, we need
\[
\deg^{0}([\alpha]_{\vec s})=d/2+\#\{i:s_i<0\}=0.
\]
Therefore, we must have
\[
d=0, \text{ and } \#\{i:s_i<0\}=0.
\]
Hence, we have a canonical basis of $QH^0(X_{D,\infty})$ given by identity classes of $\mathfrak H_{\vec s}$ when $s_i\geq 0$ for all $i\in\{1,\ldots,n\}$.
So there is a bijection between such classes and integer points of $B(\mathbb Z)$.
Hence there is a bijection between this canonical basis of $QH^0(X_{D,\infty})$, denoted by $[1]_{p}$, and prime fundamental classes $\vartheta_p\in QH_{\log}^0(X,D)$. We can also use theta functions $\vartheta$ as the canonical basis of $QH^0(X_{D,\infty})$. 
Then we can write
\[
QH^0(X_{D,\infty})=\bigoplus_{p\in B(\mathbb Z)}\mathbb C[\![ \on{NE}(X)]\!]\vartheta_p
\]
as a free $\mathbb C[\![ \on{NE}(X)]\!]$-module.

One can replace the log invariant $N^\beta_{p_1,\ldots, p_m,0}$ defined in (\ref{log-inv}) by the corresponding invariant of $X_{D,\infty}$ (with the same input data), denoted by $N^{\on{orb},\beta}_{p_1,\ldots, p_m,0}$. The product $\vartheta_{p_1}\star \vartheta_{p_2}$ is simply replaced by the restriction of the small relative quantum product $[1]_{p_1}\star_{\on{sm}}[1]_{p_2}$ to $QH^0(X_{D,\infty})$. We denote this product by $\vartheta_{p_1}\star_{\on{orb}} \vartheta_{p_2}$. The structure constant $N^{\on{orb}}_{p_1,p_2,-r}$ is defined as the invariant of $X_{D,\infty}$ with two ``inputs'' with positive contact orders given by $p_1, p_2\in B(\mathbb Z)$, one ``output'' with negative contact order given by $-r$ such that $r\in B(\mathbb Z)$, and a point constraint for the punctured point. Namely,
\begin{align}\label{def-stru-const}
    N^{\on{orb},\beta}_{p_1,p_2,-r}=\langle [1]_{p_1},[1]_{p_2},[pt]_{-r}\rangle_{0,3,\beta}^{X_{D,\infty}}.
\end{align}
The corresponding punctured invariants are structure constants considered in \cite{GS19}\footnote{The notation in \cite{GS19} is $N^{\beta}_{p_1,p_2,r}$ which is slightly different from what we use here.}. Similarly, we define
\[
N^{\on{orb},\beta}_{p_1,\ldots, p_m,0}:=\left\langle [1]_{p_1},\ldots,[1]_{p_m},[pt]_{0}\bar{\psi}^{m-2}\right\rangle^{X_{D,\infty}}_{0,m+1,\beta}.
\]

In the next lemma (see also \cite[Lemma 2.1]{GS18} for the corresponding lemma for punctured invariants), we will show that the virtual dimension constraint implies that the number $N_{p_1,p_2,-r}^{\on{orb},\beta}=0$ unless $\int_\beta[K_X+D]=0$. Similarly, for $N^{\on{orb},\beta}_{p_1,\ldots, p_m,0}$, which will appear in Theorem \ref{thm-frob}.

\begin{lem}
 For $p,q,r\in B(\mathbb Z)$,
 \begin{align*}
    N^{\on{orb},\beta}_{p_1,p_2,-r}=0
\end{align*}
if $\int_\beta[K_X+D]\neq 0$. 
\end{lem} 
\begin{proof}
Since $r\in B(\mathbb Z)$, contact orders at the third marking, represented by $-r$, are non-positive with each divisor $D_i$. Then the definition of $\deg^0$ in (\ref{deg-0}) implies that
\[
\deg^0([pt]_{-r})=\dim_{\mathbb C} X.
\]
The virtual dimension constraint (\ref{vir-dim}) is
\begin{align*}
    \dim_{\mathbb C}X-3+3-\int_\beta[K_X+D]&=\deg^0([pt]_{-r}),\\
    \text{i.e. }\quad \int_\beta[K_X+D]&=0.
\end{align*}
\end{proof}

Note that the restriction of the quantum product may involve infinite sums. For the finiteness of the product rule, we will follow \cite{GS19}. Let $P\subset H_2(X)$ be a finitely generated submonoid, containing all effective curve classes and the group of invertible elements $P^\times$ of $P$ coincides with the torsion part of $H_2(X)$. Let $I\subset P$ be a monoid ideal such that $P\setminus I$ is finite. That is,
\begin{align}\label{S-I}
S_I:=\mathbb C[P]/I
\end{align}
is Artinian. Then one can define
\begin{align}\label{R-I}
R_I:=\bigoplus_{p\in B(\mathbb Z)}S_I\vartheta_p,
\end{align}
which is a free $S_I$-module.

Replacing punctured invariants by orbifold invariants, we write the product as
\begin{align}\label{equ-prod}
\vartheta_{p_1}\star_{\on{orb}} \vartheta_{p_2}=\sum_{\beta \in P\setminus I,r\in B(\mathbb Z)}N^{\on{orb},\beta}_{p_1,p_2,-r}q^\beta\vartheta_r.
\end{align}

\begin{thm}\label{thm-stru-const}
When $(K_X+D)$ is nef or anti-nef, the structure constants $N^{\on{orb},\beta}_{p_1,p_2,-r}$ define a commutative, associative $S_I$-algebra structure on $R_I$ with unit given by $\vartheta_0$.
\end{thm}

We will refer to $R_I$ as \emph{mirror algebra}.

\begin{proof}

The finiteness of the product rule follows directly from the definition of the structure constants $N^{\on{orb},\beta}_{p_1,p_2,-r}$ and the fact that $P\setminus I$ is finite.

The commutativity is straightforward. It follows from the fact that the structure constants are Gromov-Witten invariants of $X_{D,\infty}$ which satisfy
\[
N^{\on{orb},\beta}_{p_1,p_2,-r}=N^{\on{orb},\beta}_{p_2,p_1,-r}.
\]

The fact that the class $\vartheta_0$ is the unit can be rephrased in terms of the invariants $N^{\on{orb},\beta}_{p_1,p_2,-r}$ as follows. For $p\in B(\mathbb Z)$,
\[
\begin{split}
    N^{\on{orb},\beta}_{0,p,-r}=
    \begin{cases}
    0 & \beta \neq 0 \text{ or } p\neq r,\\
    1 & \beta=0, p=r.
    \end{cases}
\end{split}
\]
But this is a direct consequence of the string equation (\ref{string-equ}).

The associativity for the relative quantum product follows from the WDVV equation (\ref{wdvv}). However, as mentioned in \cite{GS19}, the product rule that we consider here is only a truncation (restriction) of the actual product rule for relative quantum cohomology, so the associativity is not preserved in general. Here comes the assumption that $\pm(K_X+D)$ is nef. Under this assumption, we will show that the associativity is preserved.

For the associativity, we need to prove that
\[
(\vartheta_{p_1}\star_{\on{orb}} \vartheta_{p_2})\star_{\on{orb}} \vartheta_{p_3}=\vartheta_{p_1}\star_{\on{orb}} (\vartheta_{p_2}\star_{\on{orb}} \vartheta_{p_3}).
\]
Since 
\begin{align*}
(\vartheta_{p_1}\star_{\on{orb}} \vartheta_{p_2})\star_{\on{orb}} \vartheta_{p_3}&=\left(\sum_{\beta_1\in P\setminus I,s\in B(\mathbb Z)}N^{\on{orb},\beta_1}_{p_1,p_2,-s}q^{\beta_1}\vartheta_s\right)\star_{\on{orb}}\vartheta_{p_3}\\
&=\sum_{\beta_1,\beta_2\in P\setminus I,s\in B(\mathbb Z)}N^{\on{orb},\beta_1}_{p_1,p_2,-s}N^{\on{orb},\beta_2}_{s,p_3,-r}q^{\beta_1+\beta_2}\vartheta_r
\end{align*}
and
\begin{align*}
\vartheta_{p_1}\star_{\on{orb}} (\vartheta_{p_2}\star_{\on{orb}} \vartheta_{p_3})&=\vartheta_{p_1}\star_{\on{orb}}\left(\sum_{\beta_1\in P\setminus I,s\in B(\mathbb Z)}N^{\on{orb},\beta_1}_{p_2,p_3,-s}q^{\beta_1}\vartheta_s\right)\\
&=\sum_{\beta_1,\beta_2\in P\setminus I,s\in B(\mathbb Z)}N^{\on{orb},\beta_1}_{p_2,p_3,-s}N^{\on{orb},\beta_2}_{s,p_1,-r}q^{\beta_1+\beta_2}\vartheta_r.
\end{align*}
Therefore, we just need to prove
\begin{align}\label{equ-assoc}
\sum_{\substack{\beta_1+\beta_2=\beta\in P\setminus I\\s\in B(\mathbb Z)}}N^{\on{orb},\beta_1}_{p_1,p_2,-s}N^{\on{orb},\beta_2}_{s,p_3,-r}=\sum_{\substack{\beta_1+\beta_2=\beta\in P\setminus I\\s\in B(\mathbb Z)}}N^{\on{orb},\beta_1}_{p_2,p_3,-s}N^{\on{orb},\beta_2}_{s,p_1,-r},
\end{align}
where each sum is over all possible splitting of $\beta_1+\beta_2=\beta$ and all $s\in B(\mathbb Z)$. However, this is {\em not} the WDVV equation (\ref{wdvv})! The WDVV equation is of the following form with extra terms in each sum. We need to use the bracket notation to write it down:
\begin{align}\label{wdvv-frob}
   &\sum_{\substack{\beta_1+\beta_2=\beta\in H_2(X)\\ \vec s\in (\mathbb Z)^n,k}} \left\langle [1]_{p_1},[1]_{p_2},\tilde{T}_{-\vec s,k}\right\rangle_{0,3,\beta_1} \left\langle  \tilde{T}_{\vec s}^k, [1]_{p_3}, [pt]_{-r}\right\rangle_{0,3,\beta_2}\\
  \notag = &  \sum_{\substack{\beta_1+\beta_2=\beta\in H_2(X)\\ \vec s\in (\mathbb Z)^n,k}} \left\langle [1]_{p_2},[1]_{p_3},\tilde{T}_{-\vec s,k}\right\rangle_{0,3,\beta_1} \left\langle  \tilde{T}_{\vec s}^k, [1]_{p_1}, [pt]_{-r}\right\rangle_{0,3,\beta_2},
\end{align}
where $p_1,p_2,r\in B(\mathbb Z)$; each sum is over all splittings of $\beta_1+\beta_2=\beta$, all indices $\vec s, k$ of basis. We will see that extra terms in the WDVV equation vanish under the assumption that $\pm(K_X+D)$ is nef.

When $-K_X-D$ is nef, we consider the invariant $\left\langle [1]_{p_1},[1]_{p_2},\tilde{T}_{-\vec s,k}\right\rangle_{0,3,\beta_1}$ in (\ref{wdvv-frob}). The virtual dimension constraint (\ref{vir-dim}) becomes
\begin{align}\label{vir-dim-nef}
\notag\dim_{\mathbb C}X-3+3+\int_{\beta_1} [-K_X-D]=\deg^0(\tilde{T}_{-\vec s,k})\\
\dim_{\mathbb C}X+\int_{\beta_1} [-K_X-D]=\deg^0(\tilde{T}_{-\vec s,k}).
\end{align}
Let $\deg([\alpha])$ be the real degree of $\alpha\in H^*(D_I)$ for $I\subseteq \{1,\ldots,n\}$. Recall that 
\[
\deg^0(\tilde{T}_{-\vec s,k})=\deg(\tilde{T}_{-\vec s,k})/2+\#\{i:-s_i<0\}.
\]
Since
\[
\deg(\tilde{T}_{-\vec s,k})/2\leq \dim_{\mathbb C}D_{I_{\vec s}}\leq \dim_{\mathbb C}X-\#\{i:-s_i\neq 0\},
\]
we have 
\[
\deg^0(\tilde{T}_{-\vec s,k})\leq \dim_{\mathbb C}X-\#\{i:-s_i\neq 0\}+\#\{i:-s_i<0\}= \dim_{\mathbb C}X-\#\{i:-s_i>0\}.
\]
Therefore, if $\#\{i:-s_i>0\}>0$, we must have
\[
\deg^0(\tilde{T}_{-\vec s,k})<\dim_{\mathbb C}X.
\]
On the other hand, $-K_X-D$ is nef implies that
\[
\dim_{\mathbb C}X+\int_{\beta_1} ([-K_X-D]) \geq \dim_{\mathbb C}X.
\]
Hence, the virtual dimension constraint (\ref{vir-dim-nef}) does not hold unless $\#\{i:-s_i>0\}=0$, in other words, $-s_i\leq 0$ for all $i\in\{1,\ldots,n\}$. Furthermore, we must have
\[
\tilde{T}_{-\vec s,k}=[pt]_{-s}, \text{ for some } s\in B(\mathbb Z).
\]
It implies that LHS of (\ref{equ-assoc})= LHS of (\ref{wdvv-frob}) modulo $I$. The same argument implies that RHS of (\ref{equ-assoc})= RHS of (\ref{wdvv-frob}) modulo $I$. This completes the case when $-K_X-D$ is nef.

When $K_X+D$ is nef, we consider the invariant $\left\langle  \tilde{T}_{\vec s}^k, [1]_{p_3}, [pt]_{-r}\right\rangle_{0,3,\beta_2}$ in (\ref{wdvv-frob}). The virtual dimension constraint (\ref{vir-dim}) becomes
\[
\dim_{\mathbb C}X-\int_{\beta_2} [K_x+D]=\deg^0(\tilde{T}_{\vec s}^k)+\deg^0([pt]_{-r}).
\]
Since $r\in B(\mathbb Z)$, contact orders represented by $-r$ are non-positive. The definition of $\deg^0$ in (\ref{deg-0}) implies that
\[
\deg^0([pt]_{-r})=\dim_{\mathbb C} X.
\]
Then $-\int_{\beta_2} [K_x+D]\leq 0$ implies that
\[
\deg^0(\tilde{T}_{\vec s}^k)\leq 0.
\]
Therefore, we must have
\[
\deg^0(\tilde{T}_{\vec s}^k):=\deg(\tilde{T}_{\vec s}^k)+\#\{i:s_i<0\}=0.
\]
Hence, $\#\{i:s_i<0\}=0$ and
\[
\tilde{T}_{\vec s}^k=[1]_{s}, \text{ for some } s\in B(\mathbb Z).
\]
So LHS of (\ref{equ-assoc})= LHS of (\ref{wdvv-frob}) modulo $I$. The same argument implies that RHS of (\ref{equ-assoc})=RHS of (\ref{wdvv-frob}) modulo $I$. This completes the case when $K_X+D$ is nef, hence, completes the proof of the theorem.
\end{proof}

\subsection{The Frobenius structure conjecture}

\begin{thm}\label{thm-frob}
When $(K_X+D)$ is nef or anti-nef, Conjecture \ref{conj-frob} holds for $QH^0(X_{D,\infty})$.
\end{thm}
\begin{proof}
The case of $m=2$ directly follows from the definition of our structure constants $N_{p_1,p_2,0}^{\on{orb},\beta}$. The case of $m\geq3$ can be proved using TRR (\ref{trr}). 

We need to show that
\[
\sum_{\beta\in H_2(X)} N^{\on{orb},\beta}_{p_1,\ldots, p_m,0}q^\beta
\]
coincides with the coefficient of $\vartheta_0$ in the product $\vartheta_{p_1}\star_{\on{orb}} \cdots \star_{\on{orb}}\vartheta_{p_m}$. Recall that
\[
N^{\on{orb},\beta}_{p_1,\ldots, p_m,0}:=\left\langle [1]_{p_1},\ldots,[1]_{p_m},[pt]_{0}\bar{\psi}^{m-2}\right\rangle^{X_{D,\infty}}_{0,m+1,\beta}.
\]
Similar to absolute Gromov-Witten theory, TRR (\ref{trr}) can be used to remove the descendant class $\bar{\psi}$. We have
\begin{align}\label{trr-frob}
    N^{\on{orb},\beta}_{p_1,\ldots, p_m,0}=\sum \left\langle [pt]_0\bar{\psi}^{m-3},\prod_{j\in S_1}[1]_{p_j},\tilde{T}_{\vec s,k}\right\rangle \left\langle \tilde{T}_{-\vec s}^k, [1]_{p_1}, [1]_{p_2}, \prod_{j\in S_2}[1]_{p_j}\right\rangle,
\end{align}
where the sum is over all splittings of $\beta_1+\beta_2=\beta$, all indices $\vec s, k$ of basis, and all splittings of disjoint sets $S_1$, $S_2$ with $S_1\cup S_2=\{3,\ldots,m\}$. We will show that some terms in (\ref{trr-frob}) vanish and the RHS of (\ref{trr-frob}) coincide with the coefficient of $\vartheta_0$ of the product.

When $-K_X-D$ is nef, we consider the invariant $\left\langle \tilde{T}_{-\vec s}^k, [1]_{p_1}, [1]_{p_2}, \prod_{j\in S_2}[1]_{p_j}\right\rangle$ in (\ref{trr-frob}). The virtual dimension constraint (\ref{vir-dim}) is
\begin{align}\label{vir-dim-nef-trr}
\dim_{\mathbb C}X+|S_2|+\int_{\beta_2}[-K_X-D]=\deg^0(\tilde{T}_{-\vec s}^k).
\end{align}
Note that
\begin{align*}
\deg^0(\tilde{T}_{-\vec s}^k):=&\deg(\tilde{T}_{-\vec s}^k)+\#\{i:-s_i<0\}\\
\leq & \dim_{\mathbb C}X-\#\{i:-s_i\neq 0\}+\#\{i:-s_i<0\}\\
 = &\dim_{\mathbb C}X-\#\{i:-s_i>0\}\\
 \leq &\dim_{\mathbb C }X.
\end{align*}
On the other hand, $-K_X-D$ is nef implies 
\[
\dim_{\mathbb C}X+|S_2|+\int_{\beta_2}[-K_X-D]\geq \dim_{\mathbb C}X.
\]
Therefore, the equality (\ref{vir-dim-nef-trr}) does not hold unless
\[
|S_2|=0, \quad \int_{\beta_2}[-K_X-D]=0, \quad \#\{i:-s_i>0\}=0
\]
and
\[
\tilde{T}_{-\vec s}^k=[pt]_{-s}, \text{ for some } s\in B(\mathbb Z). 
\]
Therefore (\ref{trr-frob}) becomes
\begin{align*}
    N^{\on{orb},\beta}_{p_1,\ldots, p_m,0}=&\sum_{\beta_1+\beta_2=\beta\in H_2(X),s\in B(\mathbb Z)} \left\langle [pt]_0\bar{\psi}^{m-3},[1]_{p_3},\ldots,[1]_{p_m},[1]_{s}\right\rangle \left\langle [pt]_{-s}, [1]_{p_1}, [1]_{p_2}\right\rangle\\
    =& \sum_{\beta_1+\beta_2=\beta\in H_2(X),s\in B(\mathbb Z)} N^{\on{orb},\beta_1}_{s,p_3,\ldots, p_m,0}N^{\on{orb},\beta_2}_{p_1,p_2,-s}.
\end{align*}
Repeat this process $(m-3)$-times, we get
\[
N^{\on{orb},\beta}_{p_1,\ldots, p_m,0}=\sum_{\sum_{i=1}^{m-1} \beta_i=\beta\in H_2(X),s_i\in B(\mathbb Z)}N^{\on{orb},\beta_2}_{p_1,p_2,-s_1}N^{\on{orb},\beta_2}_{s_1,p_3,-s_2}\cdots N^{\on{orb},\beta_{m-1}}_{s_{m-2},p_m,0}.
\]
The right-hand side is precisely the coefficient of $\vartheta_0$ of $\vartheta_{p_1}\star_{\on{orb}} \cdots \star_{\on{orb}}\vartheta_{p_m}$ by definition. This completes the case when $-K_X-D$ is nef.

When $K_X+D$ is nef, we consider the invariant $\left\langle [pt]_0\bar{\psi}^{m-3},\prod_{j\in S_1}[1]_{p_j},\tilde{T}_{\vec s,k}\right\rangle$ in (\ref{trr-frob}). The virtual dimension constraint (\ref{vir-dim}) is
\begin{align}\label{vir-dim-trr-anti-nef}
\dim_{\mathbb C}X-3+2+|S_1|+\int_{\beta_1}[-K_X-D]=\dim_{\mathbb C}X+m-3+\deg^{0}(\tilde{T}_{\vec s, k}).
\end{align}
Since $|S_1|\leq m-2$ and $K_X+D$ is nef, we have
\[
\dim_{\mathbb C}X-3+2+|S_1|+\int_{\beta_1}[-K_X-D]\leq \dim_{\mathbb C}X-1+m-2=\dim_{\mathbb C}X+m-3.
\]
On the other hand, 
\[
\dim_{\mathbb C}X+m-3+\deg^{0}(\tilde{T}_{\vec s, k})\geq \dim_{\mathbb C}X+m-3.
\]
Therefore, the equality (\ref{vir-dim-trr-anti-nef}) does not hold unless
\[
|S_1|=m-2, \quad, \int_{\beta_1}[-K_X-D]=0, \quad \#\{i:s_i<0\}=0,
\]
and
\[
\tilde{T}_{\vec s, k}=[1]_{s}, \text{ for some }s\in B(\mathbb Z).
\]
Hence (\ref{trr-frob}) becomes
\begin{align*}
    N^{\on{orb},\beta}_{p_1,\ldots, p_m,0}=&\sum_{\beta_1+\beta_2=\beta\in H_2(X),s\in B(\mathbb Z)} \left\langle [pt]_0\bar{\psi}^{m-3},[1]_{p_3},\ldots,[1]_{p_m},[1]_{s}\right\rangle \left\langle [pt]_{-s}, [1]_{p_1}, [1]_{p_2}\right\rangle\\
    =& \sum_{\beta_1+\beta_2=\beta\in H_2(X),s\in B(\mathbb Z)} N^{\on{orb},\beta_1}_{s,p_3,\ldots, p_m,0}N^{\on{orb},\beta_2}_{p_1,p_2,-s}.
\end{align*}
We again repeat this process $(m-3)$-times to have
\[
N^{\on{orb},\beta}_{p_1,\ldots, p_m,0}=\sum_{\sum_{i=1}^{m-1} \beta_i=\beta\in H_2(X),s_i\in B(\mathbb Z)}N^{\on{orb},\beta_2}_{p_1,p_2,-s_1}N^{\on{orb},\beta_2}_{s_1,p_3,-s_2}\cdots N^{\on{orb},\beta_{m-1}}_{s_{m-2},p_m,0},
\]
where the right-hand side is precisely the coefficient of $\vartheta_0$ of $\vartheta_{p_1}\star_{\on{orb}} \cdots \star_{\on{orb}}\vartheta_{p_m}$. This completes the proof of the case when $K_X+D$ is nef, hence completes the proof of the theorem. 
\end{proof}

\subsection{Mirror construction}\label{sec:mirror-constr}
With the mirror algebra $R_I$, one can construct the mirror following the Gross-Siebert program. We will follow the construction in \cite{GS18} and \cite{GS19}. 

Let $(X,D)$ be a log Calabi-Yau pair and $B$ be pure-dimensional with $\dim_{\mathbb R}B=\dim_{\mathbb C} X$. One can define families of schemes
\[
\on{Spec}R_I\rightarrow \on{Spec} S_I.
\]
Taking the direct limit of this families of schemes, one obtains a formal flat family of affine schemes
\begin{align}\label{mirror-log-cy}
    \check{\mathfrak X}\rightarrow \on{Spf}\widehat{\mathbb C[P]},
\end{align}
where $\widehat{\mathbb C[P]}$ is the completion of $\mathbb C [P]$ with respect to the maximal ideal $P\setminus P^\times$. The family $(\ref{mirror-log-cy})$ can be viewed as the mirror family to $X\setminus D$.

Next, we consider mirrors to a degeneration of Calabi-Yau manifolds 
\[
g: X\rightarrow S,
\]
so that $D=g^{-1}(0)$ set-theoretically. One can define the ring
\[
\widehat{R}=\oplus_{p\in B(\mathbb Z)}\widehat{\mathbb C[P]}\vartheta_p.
\]
The multiplication will always be a finite sum as mentioned in \cite[Construction 1.19]{GS19}. Furthermore, $\widehat{R}$ carries an associative $\widehat{\mathbb C[P]}$-algebra structure with a natural grading. When $\dim_{\mathbb R}B=\dim_{\mathbb C} X$, the mirror family is defined to be the flat family
\[
\check{\mathcal X}=\on{Proj} \widehat{R}\rightarrow \on{Spec} \widehat{\mathbb C[P]}.
\]

\begin{rem}
\cite{GS19} actually described the mirrors in a more general setting. One can also try to construct mirrors following \cite{GS19} using the more general setting, but with invariants of $X_{D,\infty}$. 
We do not repeat these constructions here and refer readers to \cite{GS19} for more details. An interesting question to ask is that if our construction agrees with the construction in \cite{GS19}. We plan to study this question in the future.
\end{rem}

\section{A partial cohomological field theory}\label{sec:CohFT}

In this section, we show that the formal Gromov-Witten theory of infinite root stacks form a partial cohomological field theory (partial CohFT) in the sense of \cite{LRZ}. This generalizes the result of \cite[Section 3.5]{FWY19} to infinite root stacks with simple normal crossing divisors. We first provide a brief review of the CohFT.

Let $\overline {M}_{g,m}$ be the moduli space of genus $g$, $m$-pointed stable curves. We assume that $2g-2+m>0$. There are several canonical morphisms between moduli space $\overline{M}_{g,m}$ of stable curves. 
\begin{itemize}
    \item Forgetful morphisms
    \[
    \pi:\overline{M}_{g,m+1}\rightarrow \overline{M}_{g,m}
    \]
    obtained by forgetting the last marking of $(m+1)$-pointed, genus $g$ curves in $\overline{M}_{g,m+1}$.
    \item Morphisms of gluing the loops
    \[
    \rho_l: \overline{M}_{g,m+2}\rightarrow \overline{M}_{g+1,m}
    \]
    obtained by identifying the last two markings of the $(m+2)$-pointed, genus $g$ curves in $ \overline{M}_{g,m+2}$.
    \item Morphisms of gluing the trees
    \[
    \rho_t:\overline{M}_{g_1,m_1+1}\times \overline{M}_{g_2,m_2+1}\rightarrow \overline{M}_{g_1+g_2,m_1+m_2}
    \]
    obtained by identifying the last markings of separate pointed curves in $\overline{M}_{g_1,m_1+1}\times \overline{M}_{g_2,m_2+1}$.
\end{itemize}

The state space $H$ is a graded vector space with a non-degenerate pairing $\langle -,-\rangle$ and a distinguished element $1\in H$. Given a basis $\{e_i\}$, let $\eta_{jk}=\langle e_j,e_k\rangle$ and $(\eta^{jk})=(\eta_{jk})^{-1}$.

A cohomological field theory (CohFT) is a collection of homomorphisms
\[
\Omega_{g,m}: H^{\otimes m}\rightarrow H^*(\overline{M}_{g,m},\mathbb Q)
\]
satisfying the following axioms:
\begin{itemize}
    \item The element $\Omega_{g,m}$ is invariant under the natural action of symmetric group $S_m$.
    \item For all $\alpha_i\in H$, $\Omega_{g,m}$ satisfies
    \[
    \Omega_{g,m+1}(\alpha_1,\ldots,\alpha_m,1)=\pi^*\Omega_{g,m}(\alpha_1,\ldots,\alpha_m).
    \]
    \item The splitting axiom:
    \begin{align*}
    &\rho^*_t\Omega_{g_1+g_2,m_1+m_2}(\alpha_1,\ldots,\alpha_{m_1+m_2})=\\
    &\sum_{j,k}\eta^{jk}\Omega_{g_1,m_1}(\alpha_1,\ldots,\alpha_{m_1},e_j)\otimes \Omega_{g_2,m_2}(\alpha_{m_1+1},\ldots,\alpha_{m_1+m_2},e_k),
    \end{align*}
    for all $\alpha_i\in H$.
    \item The loop axiom:
    \[
    \rho_l^*\Omega_{g+1,m}(\alpha_1,\ldots,\alpha_m)=\sum_{j,k}\eta^{jk}\Omega_{g,m+2}(\alpha_1,\ldots,\alpha_m,e_j,e_k),
    \]
    for all $\alpha_i\in H$. In addition, the equality
    \[
    \Omega_{0,3}(v_1,v_2,1)=\langle v_1,v_2\rangle
    \]
    holds for all $v_1,v_2\in H$.
\end{itemize}

\begin{defn}[\cite{LRZ}, Definition 2.7]
If the collection $\{\Omega_{g,m}\}$ satisfies all the axioms except for the loop axiom, we call it a {\em partial CohFT}.
\end{defn}
We also refer to \cite[Section 3]{BR} for more discussions of infinite rank partial CohFT.

Recall that, for Gromov-Witten theory of infinite root stacks, the ring of insertions is $\mathfrak H$ defined in Section \ref{sec:state-space}. Let
\[
\pi: \bM_{g,m}(X,\beta)\times_{X^m}\left(D_{I_{\vec s^1}}\times\cdots \times D_{I_{\vec s^m}}\right)\rightarrow \bM_{g,m}
\]
be the forgetful map.

\begin{defn}
Given elements $[\alpha_1],\ldots,[\alpha_m]\in \mathfrak H$, the Gromov-Witten class for infinite root stacks is defined as 
\[
\Omega_{g,m,\beta}^{X_{D,\infty}}([\alpha_1],\ldots,[\alpha_m])=\pi_*\left(\prod_{j=1}^m \overline{\on{ev}}_j^*([\alpha_j])\cap \left[\bM_{g,m}(X_{D,\infty},\beta)\right]^{\on{vir}}\right)\in H^*(\bM_{g,m},\mathbb Q),
\]
where contact orders are specified by insertions. We then define the class
\[
\Omega_{g,m}^{X_{D,\infty}}([\alpha_1],\ldots,[\alpha_m])=\sum_{\beta\in H_2(X,\mathbb Q)}\Omega_{g,m,\beta}^{X_{D,\infty}}([\alpha_1],\ldots,[\alpha_m])q^\beta.
\]
\end{defn}

It is straightforward to check that $\Omega_{g,m}^{X_{D,\infty}}$ satisfies the first two axioms of CohFT. The proof of the splitting axiom is parallel to the proof in \cite[Theorem 3.16]{FWY19}. Therefore, we conclude that
\begin{thm}
$\Omega_{g,m}^{X_{D,\infty}}$ forms a partial CohFT.
\end{thm}

It is already known in \cite{FWY19} that the loop axiom does not hold for relative Gromov-Witten theory. Therefore, it does not hold for the formal Gromov-Witten theory of infinite root stacks. It would be interesting to find a replacement of the loop axiom. Some results along this direction has been proved in \cite{You20} by studying orbifold Gromov-Witten invariants of finite root stacks with mid-ages.

\bibliographystyle{amsxport}

\end{document}